\documentclass[a4paper,12pt]{amsart}
\usepackage[a4paper,margin=1cm]{geometry}
\usepackage[T1]{fontenc}
\usepackage{amsmath,enumerate,bm,bbm,color,fullpage,graphicx}
\usepackage{amsthm}
\usepackage{amssymb}
\usepackage{amsbsy}
\usepackage{amsfonts}
\usepackage{amstext}
\usepackage{amscd}

\numberwithin{equation}{section}
\theoremstyle{plain}

\newtheorem{theorem}{Theorem}[section]
\newtheorem{corollary}[theorem]{Corollary}
\newtheorem{lemma}[theorem]{Lemma}
\newtheorem{proposition}[theorem]{Proposition}
\theoremstyle{definition}
\newtheorem{definition}[theorem]{Definition}
\newtheorem{remark}[theorem]{Remark}
\newtheorem{example}[theorem]{Example}

\newtheorem{problem}[theorem]{Problem}

\def\C{{\mathbb C}} 
\def\R{{\mathbb R}} 
\def\N{{\mathbb N}} 
\def\bE{{\mathbb E}} 

\def\cS{{\mathcal S}} 
\def\cA{{\mathcal A}} 
\def\cB{{\mathcal B}} 
\def\cC{{\mathcal C}} 
\def\cE{{\mathcal E}} 

\def\Wg{\mathrm{Wg}} 
\def\Cat{{\rm Cat}} 
\def\Moeb{{\rm M\text{\"o}b}} 

\def\NC{\mathcal{NC}} 
\def\pair{{\rm pair}} 
\def\bfi{{\bf i}} 
\def\up{{\rm asc}} 
\def\insing{{\rm insing}} 

\def\tr{{\rm tr}} 
\def\Tr{{\rm Tr}} 
\def\diag{\,{\rm diag}} 
\def\EV{\,{\rm EV}} 

\def\alg{\mathrm{alg}} 
\def\ideal{{\rm Ideal}} 
\def\Dom{{\rm D}} 
\def\CB{C_{0,\rm b}(\R)} 
\def\CBI{C_{0,\rm b}^\infty(\R)} 
\def\ri{{\rm i}} 
\def\d{{\rm d}} 
\def\ep{\varepsilon} 
\def\tri{\trianglerighteq} 
\def\be{\begin{equation}}
\def\ee{\end{equation}}

\title{Free probability for purely discrete eigenvalues of random matrices}
\author{B. Collins}
\address{
Department of Mathematics \\
Graduate School of Science\\
Kyoto University\\
Sakyo, Kyoto 606-8502\\
Japan
} 
\email{collins@math.kyoto-u.ac.jp}
\author{T. Hasebe}
\address{
Department of Mathematics \\
 Hokkaido University\\
 Kita 10, Nishi 8, Kita-ku \\
 Sapporo 060-0810 \\
 Japan\\
} 
\email{thasebe@math.sci.hokudai.ac.jp}
\author{N. Sakuma}
\address{Department of Mathematics \\
	       Aichi University of Education\\
	       HIrosawa 1 Igaya-cho\\ 
	       Kariya-shi 448-854\\
	       Japan
	       } 
\email{sakuma@auecc.aichi-edu.ac.jp}
\date{\today}
\begin{document}

\maketitle

\begin{abstract}
In this paper, we study random matrix models which are
obtained as a non-commutative polynomial in random matrix variables of two kinds:
(a) a first kind which have a discrete spectrum in the limit,   
(b) a second kind which have a joint limiting distribution in Voiculescu's sense and are globally rotationally invariant.  
We assume that each monomial constituting this polynomial contains at least one variable of type (a), and
show that this random matrix model has a set of eigenvalues that almost surely converges to a deterministic set of numbers that is either finite or accumulating to only zero in the large dimension limit. 
For this purpose we define a framework (cyclic monotone independence) for analyzing discrete spectra and develop the moment method for the eigenvalues of compact (and in particular Schatten class) operators. 
We give several explicit calculations of discrete eigenvalues of our model.  

%
%
\end{abstract}

\section{Introduction}

\subsection{Background}

Free probability (see e.g.\ \cite{VDN})  is  a branch of operator algebras 
that was invented by Voiculescu for the purpose of studying properties of free group factors. 
Later Voiculescu discovered in \cite{Voi91} that free probability has also an application to the 
behavior of eigenvalues of non-commutative polynomials in independent  
large random matrices. This is one of the most striking success of free probability. 

Let $M_n(\C)$ be the set of all $n\times n$ matrices whose entries are complex values. 
When applying free probability to random matrices, the standard assumption is that a family of Hermitian matrices
$B_1(n),\ldots , B_k(n)\in M_n(\C)$ has a {\it (joint) limiting distribution} as $n\to\infty$, meaning that for any non-commutative 
$\ast$-polynomial $P$ in $k$ variables, $\tr_n(P(B_1(n), \dots, B_k(n)))$ admits a finite limit as $n\to\infty$ where
 $\tr_n$ is the normalized trace such that $\tr_{n} (I_n)=1$. 
Then Voiculescu's result \cite{Voi98} (see also \cite{Voi91}) states that if $U(n)$ is a Haar unitary random matrix, 
then with probability one, the enlarged family 
$\{B_1(n),\ldots , B_k(n), U(n)\}$ also has a 
 joint limiting distribution almost surely as $n\to\infty$, and in the limit, 
 $U(n)$ becomes free from $\{B_1(n),\ldots , B_k(n)\}$. 
 
Furthermore, Haagerup-Thorbj{\o}rnsen \cite{HaTh2005}, Male \cite{Male2012} and Collins-Male \cite{CoMa2014}
obtained versions of Voiculescu's results in the context of operator norm convergence. 
What they proved is that the family of matrices
$\{B_1(n),\ldots , B_k(n), U(n)\}$ admits a \emph{strong (joint) limiting distribution} as $n\to\infty$, which means that 
 it has a joint limiting distribution as stated before, and in addition, for any non-commutative self-adjoint $\ast$-polynomial $P$
 in $k+1$ non-commuting variables, the random matrix $P(B_1(n), \dots, B_k(n), U(n))$ 
  has no {\it outliers}, i.e.\ no eigenvalues outside the limiting support of the spectrum.
Before the above results in free probability, several `single random matrix models' were known to have strong limiting distributions; 
for example, this is the case for Wigner matrices under some assumptions, 
and in particular for Gaussian unitary ensembles (GUE) and Wishart matrices (see \cite[Theorem 2.1.22 and Bibliographical notes]{AGZ2010}). 

On the other hand,  in the last 10 years, random matrix models that do not have strong limiting distributions have become fashionable. 
The literature is abundant. We refer for examples to \cite{BBCF,BGN2011,CCDF2009,PRS2013,RS2013,Tao2013} and in particular to the pioneering work of Baik-Ben Arous-P\'ech\'e \cite{BBAP2005} and P\'ech\'e \cite{Pec2006} where outliers of finite rank deformations of Wishart matrices and of GUEs have been studied, respectively.

\subsection{Our model}
The purpose of this paper is to investigate random matrix models similar to those studied in 
\cite{BBAP2005,BBCF,BGN2011,Pec2006}, but our models admit {\it purely discrete spectra} when the dimension 
tends to infinity. In other words, our model has, as a limiting spectrum, the eigenvalues of a selfadjoint compact operator 
on a Hilbert space. 
The precise definition of our model is as follows. For simplicity, the dependence on $n$ being dropped, let 
$\{A_1, \dots, A_k\}$ be a family of $n\times n$ deterministic matrices which has a limiting joint distribution with respect 
to the {\it non-normalized trace} $\Tr_n$, i.e.\ for any non-commutative $\ast$-polynomial $P$ 
without a constant term, 
the following limit exists: 
\be\label{JointA}
\lim_{n\to\infty} \Tr_n(P(A_1, \dots, A_k)). 
\ee 
Let $\{B_1, \dots, B_\ell\}$ be a family of $n\times n$ deterministic matrices which has a limiting joint distribution with respect to $\tr_n$. 
Let $U$ be an $n\times n$ Haar unitary (we can treat several independent Haar unitaries, but for now we restrict to a single Haar unitary). We investigate the limiting eigenvalues of 
\be\label{model}
P(A_1,\dots, A_k, U B_1 U^*, \dots, U B_\ell U^*), 
\ee 
where $P$ is a $k+\ell$ variables selfadjoint non-commutative $\ast$-polynomial $P(\{x_i\}_{i=1}^k, \{y_j\}_{j=1}^\ell)$ such that $P(\{0\}_{i=1}^k, \{y_j\}_{j=1}^\ell)=0$.

\subsection{Main results}
The main results of this paper are as follows. 
\begin{enumerate}[\rm(i)]

\item We introduce and investigate \emph{cyclic monotone independence} which is a universal computation rule for mixed 
moments with respect to weights (Section \ref{sec:cyclic-monotone}) as abstraction of the formula \cite[Lemma 3.1]{Shl}. 
\vspace{2mm}

\item The pair of tuples $(\{A_i\}_{i=1}^k, \{U^\ast B_j U\}_{j=1}^\ell)$ above is asymptotically cyclic monotone with respect 
to $(\Tr_n, \tr_n)$ almost surely (Theorem \ref{ACM3}).  
\vspace{2mm}

\item The eigenvalues of \eqref{model} converge to deterministic eigenvalues of a compact operator almost surely (Corollary \ref{convergence EV}). We then extend (by functional calculus) this result to compact operators $A_1,\dots, A_k$ for which the limit \eqref{JointA} may not exist (Theorem \ref{enhancedACM}).

\vspace{2mm}

\item We compute in Theorem \ref{enhancedACM2} the limiting eigenvalues of \eqref{model} explicitly when the polynomial $P$ is of the forms  
$$
\sum_{i=1}^k x_i y_i x_i^*, \qquad \sum_{i=1}^k y_i x_i y_i^*,\qquad  x y + y x, \qquad \ri(x y -y x).  
$$
\vspace{0mm}

\item We also discuss a generalization of the model \eqref{model} when several independent Haar unitaries appear. We show the almost sure convergence of eigenvalues (Corollary \ref{G convergence EV}) and compute explicit eigenvalues for some polynomials of random matrices (Proposition \ref{enhancedACM3}). 

\end{enumerate}
Our model \eqref{model} is closely related to the recent work of Shlyakhtenko \cite{Shl} where asymptotic infinitesimal freeness was proved for $\{A_1,\dots, A_k\}$ and 
$\{U B_1 U^*, \dots, U B_\ell U^*\}$ when $A_1,\dots,A_k$ are fixed matrices of finite size. We strengthen this result with a 
self-contained proof, and then show an almost sure asymptotic convergence result. 

Shlyakhtenko gives a very interesting interpretation of his result in terms of locations of outliers (discrete spectra) and continuous spectra, assuming that the outliers exist. His arguments strongly suggest that infinitesimal freeness can be useful for outlier problems studied by Baik-Ben Arous-P\'ech\'e \cite{BBAP2005}, P\'ech\'e \cite{Pec2006} and others. On the other hand, our research is 
devoted to purely discrete spectra. In our discrete spectrum model, we are able to show that the ``outliers'' indeed exist almost surely for completely general polynomials in 
 general matrices $A_1,\dots, A_k$ converging to compact operators and rotationally invariant random matrices. 

On methodology, our model needs a new method outside the standard techniques in free probability, since our model 
\eqref{model} {\it converges to $0$} in distribution in the usual sense \cite{Shl}: 
\be
\lim_{n\to\infty} \tr_n\left(P(A_1,\dots, A_k, U B_1 U^*, \dots, U B_\ell U^*)^p \right) =0,\qquad p \in\N. 
\ee
To analyze the discrete spectrum, we develop the moment method with respect to the {\it non-normalized trace} $\Tr_n$. 
The most important point is that the convergence of moments with respect to $\Tr_n$ as $n\to\infty$ implies the pointwise 
convergence of eigenvalues. Together with this moment method, the {\it Weingarten calculus} developed in free 
probability \cite{Collins2003,CoSn2006} enables us to compute moments with respect to $\Tr_n$ and prove the pointwise 
convergence of eigenvalues. When continuous and discrete spectra are mixed, it is not obvious if our method can somehow be extended. 

\subsection{Organization of this paper}
After this introduction, Section \ref{sec:cv-eigenvalues} gathers preliminary materials in order to handle eigenvalue distributions of 
non-commutative random variables that are compact operators. In Section \ref{sec:cyclic-monotone} we introduce 
the notion, central to this paper, of cyclic monotone independence which is a special case of infinitesimal freeness, 
 in the framework of a non-commutative probability space with a tracial weight. 
Section \ref{sec:rmt} shows the almost sure asymptotic cyclic monotone independence of $\{A_1,\dots,A_k\}$ and $\{U B_1U^\ast, \dots, U B_\ell U^\ast\}$. Finally, Section \ref{sec5} provides examples of eigenvalues of our model \eqref{model} in large $n$ limit.

\section{Convergence of eigenvalues}\label{sec:cv-eigenvalues}
In this section, we consider convergence of eigenvalues in a general setting. First, we introduce an order and the 
classification of eigenvalues to prove theorems in this section.
After then, we obtain characterization of convergence of eigenvalues for Schatten class operators from a viewpoint 
of the non-normalized trace. 
The results play crucial roles to handle eigenvalues using non-commutative probability theory. 

We use the following notations in this paper (in particular in this section). 
\begin{enumerate}[\rm(1)]
\item $\CB$: The set of real-valued bounded continuous functions  on $\R$ that vanish in a neighborhood of $0$. 

\item $\|\cdot \|_{[\alpha,\beta]}$: The supremum norm on $C[\alpha,\beta]$, $-\infty<\alpha <\beta<\infty$.

\item $\CBI$: The set of functions $f\in \CB$ that are infinitely differentiable. 

\item $\Tr_H$: The trace on a separable Hilbert space $H$. When $H=M_n(\C)$, $\Tr_H$ is denoted by $\Tr_n$. 

\item $\tr_n$: The normalized trace $\frac{1}{n}\Tr_n$ on $M_n(\C)$. 

\item $\cS^p(H)$: The set of $p$-Schatten class operators on a separable Hilbert space $H$. 

\item $\|\cdot \|_p$: The $p$-Schatten norm. If $a$ is a selfadjoint compact operator with eigenvalues $\{\lambda_i\}_{i\geq1}$ then $\|a\|_p= \left(\sum_{i=1}^\infty |\lambda_i|^p\right)^{1/p}$.

\end{enumerate}

\subsection{Order for eigenvalues} 
It is useful to regard eigenvalues as a multiset. 
\begin{definition}
For a selfadjoint compact operator $a$, we denote by 
$\EV(a)$ the multiset of its eigenvalues. The disjoint union of multisets counts the multiplicity, 
e.g.\ $\{3,2,1,1,0,0,\dots\} \sqcup \{2,1,1,0,0,\dots\} = \{3,2,2,1,1,1,1,0,0,\dots\}$. 
\end{definition}

Let us make remarks about this definition. 
\begin{itemize}
\item We may also view $\EV (a)$ as a positive measure $\mu$ on $\mathbb{R}$, namely $\mu (\{\lambda\})$ is the dimension of the eigenspace of $a$ with eigenvalue $\lambda$. 
\item 
Alternatively, we may view $\EV(a)$ as the collection of all real sequences
$(x_n)$ tending to zero, quotiented by the equivalence relation
$(x_n)\sim (y_n)$ iff there exists a permutation $\sigma$ of $\mathbb{N}$ such that
$x_n=y_{\sigma (n)}$ for all $n$.
\item
Whenever needs be, we extend the notion of eigenvalues abstractly to selfadjoint elements of a $*$-algebra with a 
tracial weight even if the weight is not a trace on a separable Hilbert space. This is defined in Section \ref{sec:cyclic-monotone} 
and is related to Section \ref{sec5}. 
\end{itemize}

In order to discuss the convergence of eigenvalues, it is useful to order them in a nice way. 

\begin{definition}
 We say that a sequence of real numbers $\{r_i\}_{i=1}^\infty$ converging to $0$ is {\it properly arranged} 
 if $|r_i| \geq |r_{i+1}|$ for all $i\in\N$. Note that the proper arrangement may not be unique.  
\end{definition}
Due to the non-uniqueness of proper arrangement, it is sometimes better to decompose a sequence into the 
nonnegative part and nonpositive part. 
 The proper arrangement of the nonnegative part $\{r_i \mid i \in \N, r_i \geq 0\}$ is unique and is denoted 
 by $\{r_i^+\}_{i=1}^\infty$, and similarly $\{r_i^-\}_{i=1}^\infty$ denotes the unique proper arrangement of $\{r_i \mid i \in \N, r_i \leq 0\}$. 

From now on, we always assume that eigenvalues $\{\lambda_i\}_{i\geq1}$ are properly arranged, namely 
\begin{equation}
|\lambda_1| \geq |\lambda_2| \geq \cdots,
\end{equation}
and their nonnegative and nonpositive parts are also properly arranged uniquely,   
\begin{equation}\label{PN}
\lambda_1^+ \geq \lambda_2^+ \geq \cdots \geq 0 \geq \cdots \geq \lambda_2^- \geq \lambda_1^-.
\end{equation}

The properly arranged eigenvalues of a selfadjoint compact operator $a$ on a separable Hilbert space are denoted 
by $\{\lambda_i(a)\}_{i\geq1}$. If the dimension of the Hilbert space is finite then we understand that the index $i$ stops at 
the dimension. Instead of $\lambda_i(a)^\pm$ we use the notation $\lambda_i^\pm(a)$ for the properly arranged nonnegative 
and nonpositive parts of the eigenvalues.

\subsection{Convergence of eigenvalues}

\begin{definition}\label{def:cv-in-eigenvalues}
Let $a, a_k,  k=1,2,3,\dots$ be selfadjoint compact operators on separable Hilbert spaces $H,H_k$, respectively. 
We say that {\it $a_k$ converges to $a$ in eigenvalues} if $\lim_{k\to\infty}\lambda_i^{u}(a_k)  =\lambda_i^{u}(a)$ for any $i\in\N$ and $u\in\{+,-\}$. If a sequence stops at a finite $i$ then infinitely many $0$'s are to be added in the end. Convergence in eigenvalues is denoted by 
\[
\lim_{k\to\infty} \EV(a_k) = \EV(a). 
\]
\end{definition}
\begin{remark}
It seems also natural to define the convergence $a_k \to a$ in eigenvalues by 
\be 
\lim_{k\to\infty}\lambda_i(a_k)=\lambda_i(a) \qquad i\in\N, \tag{$\ast$}
\ee
but this is not good. For example if $\lambda_i(a_k)= \{-1, 1-\frac{1}{k}, \frac{1}{2}, \frac{1}{3}, \frac{1}{4}, \dots\}$ and $\lambda_i(a)=\{1,-1,\frac{1}{2}, \frac{1}{3}, \frac{1}{4},\dots \}$, then $a_k \to a$ in eigenvalues, but the convergence does not hold in the sense of $(\ast)$. 
\end{remark}

Note that the set of eigenvalue sequences 
\be
\{\{\lambda_i\}_{i\geq1} \subset \R: \lim_{i\to\infty}\lambda_i=0 \}
\ee
 is metrizable by 
\be\label{metric}
d(\{\lambda_i\}_{i\geq1}, \{\mu_i\}_{i\geq1}) := \sum_{i=1}^\infty \frac{1}{2^i} \frac{|\lambda_i^+ -\mu_i^+|}{1+|\lambda_i^+-\mu_i^+|} +\sum_{i=1}^\infty \frac{1}{2^i} \frac{|\lambda_i^- -\mu_i^-|}{1+|\lambda_i^--\mu_i^-|}, 
\ee
where infinitely many $0$'s are to be added in the end if the sequence $\lambda_i^\pm$ or $\mu_i^\pm$ ends at a finite $i$. This metric is compatible with the convergence in eigenvalues. 

\begin{proposition}\label{Convergence1}
  Let $a, a_k, k=1,2,3,\dots$ be selfadjoint compact operators on separable Hilbert spaces $H, H_k, k=1,2,3,\dots,$ respectively. Then the following are equivalent. 
\begin{enumerate}[\rm(1)]
\item\label{C10} $a_k$ converges to $a$ in eigenvalues (cf.\ Definition \ref{def:cv-in-eigenvalues}). 
\item\label{C20} $\lim_{k\to\infty}\Tr_{H_k}(f(a_k)) = \Tr_H(f(a))$ for any $f \in \CB$. 
\item\label{C30} $\lim_{k\to\infty}\Tr_{H_k}(f(a_k)) = \Tr_H(f(a))$ for any $f \in \CBI$. 
\end{enumerate}
\end{proposition}
\begin{proof} 
 Let $\lambda_i := \lambda_i(a)$ and $\lambda_i(k):=\lambda_i(a_k)$ for simplicity. Recall that we arrange the 
 eigenvalues in the way $\lambda_1^+ \geq \lambda_2^+ \geq \cdots$ and $\lambda_1^- \leq \lambda_2^- \leq \cdots$, 
 and similarly for $\{\lambda_i^\pm(k)\}_{i\geq1}$. 

\eqref{C10}$\Rightarrow$\eqref{C20}.  Take $f \in\CB$ then $f\equiv 0$ on $(-\delta,\delta)$ for some $\delta>0$. 
Only finitely many eigenvalues of $a$ are contained in $(-\delta,\delta)^c$, say $\lambda_i^+, i=1,2,\dots, \ell$ 
and $\lambda_i^-, i=1,2,\dots, m$. Convergence in eigenvalues implies that there exists $k_0$ such that 
$0\leq\lambda_{\ell+1}^+(k) < \delta$ for all $ k \geq k_0$. This implies that $0\leq\lambda_{i}^+(k) < \delta$ for all 
$ k \geq k_0$ and $i \geq \ell+1$. Similar facts hold for negative eigenvalues. Therefore, for sufficiently large $k$ we have 
\be
\begin{split}
\Tr_{H_k}(f(a_k))&=\sum_{i=1}^\ell f(\lambda_i^+(k)) + \sum_{i=1}^m f(\lambda_i^-(k))\\
& \overset{k\to\infty}{\longrightarrow} \sum_{i=1}^\ell f(\lambda_i^+)+  \sum_{i=1}^m f(\lambda_i^-) = \Tr_H(f(a)). 
\end{split}
\ee

\eqref{C20}$\Rightarrow$\eqref{C30}: Obvious. 

\eqref{C30}$\Rightarrow$\eqref{C10}. If we take $f \in \CBI$ such that $f\geq0$ on $\R$, $f (x)=0$ for $|x| \leq \|a\|$ and 
$f(x)=1$ for $|x| > 2 \|a\|$, then $\Tr_H(f(a))=0$. So for sufficiently large $k$ we have $\Tr_H(f(a_k)) <1$ and hence 
$a_k$ has no eigenvalues in $\{x \in \R: |x| >2 \|a\|\}$.  
This implies that the eigenvalues of $\{a_k: k \in \N\}$ are uniformly bounded, i.e.\ they are contained in a common 
interval $[-\alpha,\alpha]$. 

If $\lambda_1^+(k)$ does not converge to $\lambda_1^+$ then there exists a subsequence $(k_j)_{j\geq1}$, a real 
number $\mu_1^+ \in[0,\alpha]$ such that $\mu_1^+ \neq \lambda_1^+$ and $\lambda_1^+(k_j) \to \mu_1^+$ as $j\to\infty$. 
We derive a contradiction below. Let $\ep_1:= |\lambda_1^+ - \mu_1^+| >0$. 
\begin{itemize}
\item Case $0\leq \lambda_1^+< \mu_1^+$. We take a nonnegative function $f \in \CBI$ such that $f(x)= 1$ for 
$x\geq \mu_1^+ -\ep_1/4$ and $f(x)=0$ for $x\leq \lambda_1^++ \ep_1/4$. Then all eigenvalues of $a$ lie outside of
the support of $f$. Therefore $\Tr_{H_{k_j}}(f(a_{k_j})) \geq f(\lambda_1^+(k_j)) =1$ for large $j$, but $\Tr_H(f(a))=0$, a contradiction. 

\item Case $0\leq \mu_1^+< \lambda_1^+$. We take a nonnegative function $g \in \CBI$ such that $g(x)= 1$ for 
$x\geq \lambda_1^+-\ep_1/4$ and $g(x)=0$ for $x\leq \mu_1^+ + \ep_1/4$. Then $\Tr_{H_{k_j}}(f(a_{k_j}))=0$ for large 
$j$, but $\Tr_H(f(a))\geq f(\lambda_1^+)=1$, a contradiction. 
\end{itemize}
Thus we conclude that $\lambda_1^+(k)$ converges to $\lambda_1^+$ as $k\to\infty$. Then we go to the induction: 
Suppose that $\lambda_i^+(k) \to \lambda_i^+$ as $k\to\infty$ for every $i=1,\dots, \ell-1$ and suppose 
that $\lambda_{\ell}^+(k)$ does not converge to $\lambda_\ell^+$. Then as before there exists a subsequence 
$(k_j)_{j\geq1}$, a real number $\mu_\ell^+ \in[0,\lambda_{\ell-1}^+]$ such that $\mu_\ell^+\neq \lambda_\ell^+$ 
and $\lambda_\ell^+(k_j) \to \mu_\ell^+$ as $j\to\infty$. Let $\ep:=|\lambda_\ell^+ - \mu_\ell^+|>0$. 

\begin{itemize}
\item Case $\lambda_\ell^+ < \mu_\ell^+$. We take a nonnegative function $f \in \CBI$ such that $f(x)= 1$ 
for $x\geq \mu_\ell^+ -\ep/4$ and $f(x)=0$ for $x\leq \lambda_\ell^+ + \ep/4$. Then $\Tr_{H_{k_j}}(f(a_{k_j}))\geq\ell$ for 
large $j$, but $\Tr_H(f(a))= \sum_{i=1}^{\ell-1} f(\lambda_\ell^+)=\ell-1$, a contradiction.

\item Case  $\mu_\ell^+< \lambda_\ell^+$. We take a nonnegative function $g \in \CBI$ such that $g(x)= 1$ 
for $x\geq \lambda_\ell^+-\ep/4$ and $g(x)=0$ for $x\leq \mu_\ell^+ + \ep/4$. Then $\Tr_{H_{k_j}}(f(a_{k_j}))=\ell-1$ 
for large $j$, but $\Tr_H(f(a))\geq \ell$, a contradiction. 
\end{itemize}
Thus $\lambda_\ell^+(k)$ converges to $\lambda_\ell^+$ as $k\to\infty$. By induction we conclude that $\lambda_i^+(k)$ 
converges to $\lambda_i^+$ as $k\to\infty$ for every $i \geq1$. 
Similarly we can prove the convergence of $\lambda_i^-(k)$ to $\lambda_i^-$. 
\end{proof}

We define a notion of the distributional convergence of a tuple of compact operators, which is motivated from 
Definition \ref{def:cv-in-eigenvalues} and Proposition \ref{Convergence1}. 
\begin{definition}\label{def:cv-compact}
Given selfadjoint compact operators $a_i, a_i(n), i=1,\dots,k$ on separable Hilbert spaces $H, H_n$ respectively, we 
say that $(a_1(n) , \dots, a_k(n))$ {\it converges in compact distribution} to $(a_1,\dots, a_k)$ with respect to $\Tr_{H_n}, \Tr_H$ 
as $n\to\infty$ if for every function $f_i \in \CB, i=1,\dots, k$,  $p \in \N$ and tuple $(i_1,\dots, i_p) \in \{1,\dots, k\}^p$ we have that 
\be
\lim_{n\to\infty}\Tr_{H_n}(f_{i_1}(a_{i_1}(n))  \cdots f_{i_p}(a_{i_p}(n))) =\Tr_{H}(f_{i_1}(a_{i_1})  \cdots f_{i_p}(a_{i_p})). 
\ee
\end{definition}

\subsection{Moment method for convergence of eigenvalues}
In noncommutative probability, the moment method is an important tool to prove weak convergence of probability measures. 
Here, we show a counterpart for eigenvalues.

\begin{lemma}\label{approx} Let $\alpha >\delta >0$ and let $p \in \N$. If $f \in C^p(\R)$ and $f\equiv0$ on $[-\delta,\delta]$, 
then for any $\ep>0$, 
there exists a polynomial $P$ such that $0=P(0)=P'(0)=\cdots = P^{(p-1)}(0)$,  $\| f-P \|_{[-\alpha,\alpha]}<\ep$ and $|P(x)| \leq \ep |x|^p$ for $x\in[-\delta,\delta]$.  
\end{lemma}
\begin{proof}
Without loss of generality we may assume that $\alpha=1$. We only consider $p=2$ for simplicity; the general case is similar. 
By Weierstrass' approximation we can find a polynomial $R_0(x)$ such that $\|f''- R_0\|_{[-\alpha,\alpha]}<\ep/8$. This 
implies that $|R_0(0)| <\ep/8$, so the polynomial $R(x):=R_0(x)-R_0(0)$ satisfies that $\|f''- R_0\|_{[-\alpha,\alpha]}<\ep/4$. 
Then we define $Q(x):=  \int_0^x R(y)\,\d y$. Note that $Q(x)$ does not have a constant term. For $x \in [-1,1]$ one has that 
$|f'(x)-Q(x)| = |\int_0^x f''(y)\, \d y - \int_0^x R(y)\,\d y|  \leq \int_{-1}^1 | f''(y)- R(y)|\,\d y \leq  \ep/2$. Then we define 
$P(x):=   \int_0^x Q(y)\,\d y$, which satisfies $P(0)=P'(0)=0$. Then for any $x\in[-1,1]$ one has that  
$|f(x)-P(x)| = |\int_0^x f'(y)\, \d y - \int_0^x Q(y)\,\d y|  \leq \int_{-1}^1 | f'(y)- Q(y)|\,\d y \leq \ep$. 

For $x \in [-\delta,\delta]$, we have
\be
\begin{split}
|P(x)| &=\left| \int_0^x \d y \int_0^y R(z)\, \d z \right| \leq \int_{|y|\leq|x|} \d y \int_{|z|\leq |x|} |R(z)|\, \d z\\
& =2|x|  \int_{|z|\leq |x|} |R(z)|\, \d z \leq \ep x^2.  
\end{split}
\ee 
On the last line we have used the fact that $f''(z)= 0$ for $z\in[-\delta,\delta]$ and $\|f''-R\|_{[-\alpha,\alpha]} \leq \ep/4$. 
\end{proof}

\begin{proposition}\label{prop2.6} 
Let $a, a_k, k=1,2,3,\dots$ be selfadjoint operators on separable Hilbert spaces $H, H_k, k=1,2,3,\dots,$ respectively. 
Suppose that there exists $p \in \N$ such that 
$a \in \cS^p(H)$ and $a_k \in \cS^p(H_k)$ for all $k\in\N$. Suppose that $\Tr_{H_k}(a_k^n)\to \Tr_H(a^n)$ as $k\to\infty$ 
for any integer $n \geq p$. Then $a_k$ converges to $a$ in eigenvalues. 
\end{proposition}
\begin{proof} 
We may assume that $p$ is an even integer since $\cS^p \subset \cS^{p+1}$. 
By Proposition \ref{Convergence1}, it suffices to show that for any bounded $C^p(\R)$ function $f$ such that $f \equiv 0$ on $[-\delta,\delta]$ for some $\delta>0$, it holds that $\Tr_{H_k}(f(a_k))\to \Tr_H(f(a))$ $(k\to\infty)$.

If the dimensions of some of $H, H_k$ are finite, then the following proof is still available by adding infinitely many 0's to the eigenvalues, so let us assume that the dimensions of Hilbert spaces are all infinite. 
Let $\{\lambda_i\}_{i=1}^\infty, \{\lambda_i(k)\}_{i=1}^\infty$ be the properly arranged eigenvalues of $a, a_k$, respectively.

Let $\alpha:= \sup_{k\geq1} \|a_k\|_p<\infty$ and suppose that $\alpha>0$; otherwise $a, a_k$ are all zero elements. 
It is easy to see by Chebyshev's inequality  that 
\be\label{Chebyshev}
\#\{i \in\N: |\lambda_i(k)|>\delta\} =\#\{i \in\N: \lambda_i(k)^p >\delta^p\} \leq \frac{\Tr(a_k^p)}{\delta^p} \leq \frac{\alpha^p}{\delta^p}, 
\ee 
and similarly, $\#\{i \in\N: |\lambda_i|>\delta\} \leq \frac{\alpha^p}{\delta^p}.$ 
Therefore if we define $i_0 := [\frac{\alpha^p}{\delta^p}]\in\N$ then $|\lambda_i(k)|, |\lambda_i| \leq\delta$ for all $k\in\N$ and $i > i_0$.

Since $\|a\|_p^p, \|a_k\|_p^p \leq \alpha^p$, the eigenvalues of $a, a_k$ are all contained in the interval $[-\alpha, \alpha]$. Given $\ep>0$, 
by Lemma \ref{approx} we can find a polynomial $P$ of the form $P(x) = a_p x^p + a_{p+1} x^{p+1}+ \cdots + a_q x^q$ such that $\| f-P \|_{[-\alpha, \alpha]}<\ep/i_0$ and $|P(x)| \leq \ep \alpha^{-p} x^p$ for $x\in[-\delta,\delta]$. Then for all $k\in\N$ we have that 
\be\label{eqA}
\begin{split}
\left|\sum_{i=1}^\infty f(\lambda_i(k)) - \sum_{i=1}^\infty f(\lambda_i)\right| 
&= \left|\sum_{i=1}^{i_0} f(\lambda_i(k)) - \sum_{i=1}^{i_0} f(\lambda_i)\right| \\
&\leq  \left|\sum_{i=1}^{i_0} f(\lambda_i(k)) - \sum_{i=1}^{i_0} P(\lambda_i(k))\right| + \left|\sum_{i=1}^{i_0} P(\lambda_i(k)) - \sum_{i=1}^{i_0} P(\lambda_i)\right| \\
&\qquad + \left|\sum_{i=1}^{i_0} P(\lambda_i) - \sum_{i=1}^{i_0} f(\lambda_i)\right|  \\
&\leq 2 \ep + \left|\sum_{i=1}^{i_0} P(\lambda_i(k)) - \sum_{i=1}^{i_0} P(\lambda_i)\right|. 
\end{split}
\ee
We note here that 
\be
\sum_{i=i_0+1}^\infty |P(\lambda_i(k))| \leq \ep \alpha^{-p} \sum_{i=i_0+1}^\infty |\lambda_i(k)|^p \leq \ep \alpha^{-p} \|a_k\|_p^p \leq \ep, 
\ee
and similarly, $\sum_{i=i_0+1}^\infty |P(\lambda_i)| \leq \ep$. Therefore, 
\be\label{eqB}
 \left|\sum_{i=1}^{i_0} P(\lambda_i(k)) - \sum_{i=1}^{i_0} P(\lambda_i)\right| \leq 2\ep +  \left|\sum_{i=1}^{\infty} P(\lambda_i(k)) - \sum_{i=1}^{\infty} P(\lambda_i)\right|. 
\ee
Since $\Tr_{H_k}(a_k^n)=\sum_{i=1}^\infty \lambda_i(k)^p \to \sum_{i=1}^\infty \lambda_i^p=\Tr_H(a^n)$ for all $n\geq p$ and the coefficients of $1,x,\dots, x^{p-1}$ of $P$ are zero, there exists $k_0\in\N$ such that 
\be\label{eqC}
\left|\sum_{i=1}^{\infty} P(\lambda_i(k)) - \sum_{i=1}^{\infty} P(\lambda_i)\right| \leq \ep
\ee 
for $k\geq k_0$. From \eqref{eqA},\eqref{eqB},\eqref{eqC} we have that 
\be
\left|\sum_{i=1}^\infty f(\lambda_i(k)) - \sum_{i=1}^\infty f(\lambda_i)\right|  \leq 5 \ep,\qquad k \geq k_0, 
\ee
the conclusion. 
\end{proof} 
\begin{corollary}\label{cor unique} Let $a,b$ be selfadjoint operators on separable Hilbert spaces $H, K$, respectively, such that $a\in \cS^p(H)$ and $b \in \cS^p(K)$ for some $p \in \N$. If $\Tr_H(a^n) = \Tr_K(b^n)$ for every integer $n\geq p$ then $\EV(a)=\EV(b)$. 
\end{corollary}

Actually in Proposition \ref{prop2.6}, we need not a priori assume the existence of a limiting operator $a$.

\begin{proposition}\label{prop2.7} Let $p \in 2\N$ and $H_k, k\in\N$ be separable Hilbert spaces. Let $a_k \in \cS^{p}(H_k), k=1,2,3,\dots$ be selfadjoint.  
Suppose that the limit $\alpha_n:=\lim_{k\to\infty}\Tr_{H_k}(a_k^n) \in\R$ exists at least for all integers $n\geq p$. Then there exist a separable Hilbert space $H$ and a selfadjoint operator $a \in \cS^{p}(H)$ such that $a_k$ converges to $a$ in eigenvalues, and moreover, $\alpha_n=\Tr_{H}(a^n)$ for at least $n\geq p+1$.  
\end{proposition}
\begin{remark}
\begin{enumerate}[\rm(1)]
\item It is useless to assume that $p \in 2 \N -1$, because then the conclusion will only be $a\in \cS^{p+1}(H)$. 

\item We cannot conclude that $\alpha_p=\Tr_H(a^p)$. For example we may take as eigenvalues of $a_k$
\[
\lambda_i(k) = 
\begin{cases}
\frac{1}{i}, & 1 \leq i \leq k,\\
\frac{1}{k}, & k+1 \leq i \leq k+k^2, \\   
\frac{1}{i-k^2}, & i \geq  k +k^2+1. 
\end{cases}
\]
Then $\lambda_i:=\lim_{k\to\infty} \lambda_i(k)=\frac{1}{i}$ for all $i\in\N$, but $\sum_{i=1}^\infty \lambda_i(k)^2 = \zeta(2)+ 1 \neq \sum_{i=1}^\infty \lambda_i^2$. 
\end{enumerate}
\end{remark}
\begin{proof} Let $\{\lambda_i(k)\}_{i\geq1}$ be the properly arranged eigenvalues of $a_k$, and let 
$\alpha:= \sup_{k \in\N} \|a_k\|_{p}$.  
Since $ |\lambda_i(k)|$ are uniformly bounded by $\alpha$, there exists a subsequence $\{k(1,j)\}_{j\geq1}$ of $\N$ such that $\lambda_1(k(1,j))$ converges to some $\lambda_1$. Then we can take a subsequence $\{k(2,j)\}_{j\geq1}$ of $\{k(1,j)\}_{j\geq1}$ such that $\lambda_2(k(2,j))$ converges to some $\lambda_2$, and continue this procedure. 
Let $k_m := k(m,m)$ and $\lambda_i(m):= \lambda_i(k_m)$. By construction $\{k_m\}_{m \geq i} \subset \{k(i, j)\}_{j\geq1}$ for each $i \in \N$. Therefore $\lim_{m\to\infty} \lambda_i(m)=\lambda_i$ for every $i\in\N$. Then $\{\lambda_i\}_{i\geq1}$ is properly arranged since so is $\{\lambda_i(k)\}_{i\geq1}$.  
By Fatou's lemma 
\be
\sum_{i=1}^\infty \lambda_i^{p} \leq \varliminf_{m\to\infty} \sum_{i=1}^\infty \lambda_i(m)^{p} \leq \alpha^{p}.   
\ee
Therefore, there exists a separable Hilbert space $H$ and a selfadjoint $a \in \cS^{p}(H)$ such that $\EV(a)=\{\lambda_i\}_{i\geq1}$. 
Note that $\beta_n:= \sum_{i=1}^\infty \lambda_i^n$ is absolutely convergent for $n\geq p$ since $\cS^n(H)$ is increasing on $n\geq1$. 

We want to show that $\beta_n=\alpha_n$ for $n \geq p+1$. For any $\ep>0$, by Chebyshev's inequality \eqref{Chebyshev} 
(now we use $\ep>0$ instead of $\delta>0$), there exists $i_0\in\N$ such that $|\lambda_{i_0+1}(k)| \leq \ep$ for all $k \in \N$, 
and hence $|\lambda_{i_0+1}| \leq \ep$. For $n \geq p +1$ we have 
\be
\begin{split}
|\alpha_n-\beta_n| 
&\leq |\alpha_n - \Tr_{H_k}(a_k^n)| + |\Tr_{H_k}(a_k^n) - \beta_n|  \\
& \leq |\alpha_n - \Tr_{H_k}(a_k^n)| + \left| \sum_{i=1}^{i_0}(\lambda_i(k)^n - \lambda_i^n) \right| + \sum_{i=i_0+1}^{\infty}(|\lambda_i(k)|^n + |\lambda_i|^n)\\
& \leq |\alpha_n - \Tr_{H_k}(a_k^n)| + \left| \sum_{i=1}^{i_0}(\lambda_i(k)^n - \lambda_i^n) \right| + \ep^{n-p}\left( \sum_{i=i_0+1}^{\infty}(\lambda_i(k)^{p} + \lambda_i^{p})\right). 
\end{split}
\ee
The first two terms converge to 0 if we put $k=k_m$ and let $m$ tend to infinity. The last term is bounded by 
$2\ep^{n-p}\alpha^{p}$, so we get $|\alpha_n-\beta_n|\leq 2 \alpha^p\ep^{n-p}$ and hence $\alpha_n=\beta_n$  for $n\geq p+1$. 
By Corollary \ref{cor unique} (applied to now $p+1$), the limit eigenvalues $\{\lambda_i\}_{i\geq1}$ do not depend on 
the choice of the subsequence $\{\lambda(k_m)\}_{m\geq1}$. Since the space of eigenvalue distributions is metrizable 
(see \eqref{metric}), $a_k$ converges in eigenvalues to $a$. 
\end{proof}

\section{Cyclic monotone independence}\label{sec:cyclic-monotone}
Our aim is to analyze asymptotic behavior of discrete eigenvalues of random matrices. For this we would like to abstract the notion of the non-normalized trace $\Tr_H$ on a Hilbert space $H$. 
So we define a non-commutative measure space that replaces a state in a non-commutative probability space with a {\it weight}. 

\subsection{Algebraic non-commutative measure space} \label{sec3.1}
A \emph{non-commutative measure space} is a pair $(\cA,\omega)$ where $\cA$ is a (unital or non-unital) $\ast$-algebra over $\C$. 
Let $\omega$ be a tracial weight meaning that 
\begin{itemize} 
\item $\omega$ is defined on a (possibly non-unital) $\ast$-subalgebra $\Dom(\omega)$ of $\cA$ and $\omega: \Dom(\omega)\to\C$ is linear,  
\item $\omega$ is positive, i.e.\ $\omega(a^* a) \geq0$ for every $a\in\Dom(\omega)$, 
\item $\omega(a^*)=\overline{\omega(a)}$ for all $a \in \Dom(\omega)$, 
\item $\omega(a b) = \omega(b a)$ for all $a,b \in\Dom(\omega)$. 
 \end{itemize}
Moreover, if $\cA$ is unital, $\Dom(\omega)=\cA$ and $\omega(1_\cA)=1$ then we call $(\cA,\omega)$ {\it a non-commutative probability space}.  

Let  $(\cA,\omega)$ be a non-commutative measure space and let $a_1,\ldots ,a_k \in\Dom(\omega)$. 
The {\it distribution} of $(a_1,\ldots ,a_k)$ is the family of {\it (mixed) moments} 
\be
\{\omega(a_{i_{1}}^{\ep_1}\dots a_{i_{p}}^{\ep_p}): p\geq1, 1 \leq i_{1},\dots,i_{p}\leq k, (\ep_1,\dots, \ep_p)\in\{1,\ast\}^p\}.
\ee 
Given non-commutative measure spaces $(\cA,\omega), (\cB,\xi)$ and elements $a_1, \dots, a_k\in\Dom(\omega)$, $b_1,\dots, b_k \in \Dom(\xi)$, we say that $(a_1,\dots, a_k)$ has {\it the same distribution} as $(b_1,\dots, b_k)$ if 
\be
 \omega(a_{i_1}^{\varepsilon_{1}}\cdots a_{i_p}^{\varepsilon_{p}})= \xi(b_{i_1}^{\varepsilon_{1}}\cdots b_{i_p}^{\varepsilon_{p}})
\ee
for any choice of $p\in\N$, $1\leq i_{1},\dots,i_{p}\leq k$ and $(\varepsilon_1,\dots, \varepsilon_p)\in\{1,\ast\}^p$. 

The distribution of $(a_1,\ldots ,a_k)$ is a \emph{trace class distribution} if there exist a separable Hilbert space 
$H$ and $x_1,\ldots ,x_k \in (\cS^1(H), \Tr_H)$ (trace class operators) such that the distribution of $(a_1,\ldots ,a_k)$ is 
the same as that of $(x_1,\dots, x_k)$. In this case we define the non-zero eigenvalues of a self-adjoint $\ast$-polynomial 
$P(a_1,\dots, a_k)$ to be the eigenvalues of $P(x_1,\dots, x_k)$.

Let $(\cA,\omega), \{(\cA_n,\omega_n)\}_{n\geq1}$ be non-commutative measure spaces and $a_1,\dots, a_k \in \Dom(\omega)$, $a_1(n), \dots, a_k(n)\in\Dom(\omega_n)$. We say that  $(a_1(n),\ldots ,a_k(n))$ \emph{converges in distribution} to $(a_1,\dots, a_k)$ if 
\be
\lim_{n\to\infty} \omega_{n}(a_{i_1}(n)^{\varepsilon_{1}}\cdots a_{i_p}(n)^{\varepsilon_{p}}) =\omega(a_{i_1}^{\varepsilon_{1}}\cdots a_{i_p}^{\varepsilon_{p}})
\ee
for any choice of $p\in\N$, $1\leq i_{1},\dots,i_{p}\leq k$ and $(\varepsilon_1,\dots, \varepsilon_p)\in\{1,\ast\}^p$. 
If the distributions of selfadjoint elements $a_n \in \Dom(\omega_n)$ and $a\in\Dom(\omega)$ are trace class then we can define the concept of {\it convergence in eigenvalues} in a natural way (cf.\ Definition \ref{def:cv-in-eigenvalues}), i.e.\ $\lambda_i^u(a_n) \to \lambda_i^u(a)$ for every $i\in\N$ and $u\in\{+,-\}$.

\begin{proposition} Let $\{(\cA_n, \omega_n)\}_{n\geq1}, (\cA,\omega)$ be non-commutative measure spaces. 
Suppose that $a_i(n) \in \Dom(\omega_n), a_i \in\Dom(\omega), i=1,2,\dots, k$ have trace class distributions and 
that $(a_1(n),\dots, a_k(n))$ converges in distribution to $(a_1,\dots, a_k)$. Then for any selfadjoint non-commutative 
$\ast$-polynomial $P$ without a constant term, $P(a_1(n), \dots, a_k(n))$ converges in eigenvalues to $P(a_1,\dots, a_k)$. 
\end{proposition}
\begin{proof}
This follows from Proposition \ref{prop2.6}. 
\end{proof}
\subsection{Cyclic monotone independence} We introduce a kind of independence in a non-commutative probability 
space with a tracial weight $(\cC,\omega,\tau)$. For $\ast$-subalgebras $\cA,\cB$ of $\cC$ such that 
$1_{\cC}\in \cB$, let $\ideal_\cB(\cA)$ be the ideal generated by $\cA$ over $\cB$. More precisely, 
\be
\ideal_\cB(\cA):= {\rm span}\{b_0 a_1 b_1 \cdots a_n b_n: n \in \N, a_1,\dots, a_n \in \cA, b_0,\dots, b_n\in\cB\}, 
\ee
which is a $*$-subalgebra of $\cC$ containing $\cA$. 

\begin{definition}\label{construction} Let $(\cC,\omega,\tau)$ be a non-commutative probability space with a tracial weight $\omega$. 
\begin{enumerate}[\rm(1)]
\item Let $\cA, \cB$ be $\ast$-subalgebras of $\cC$ such that $1_\cC \in \cB$. We say that the pair $(\cA,\cB)$  
is {\it cyclically monotonically independent} or more simply {\it cyclically monotone} with respect to 
$(\omega,\tau)$ if 
\begin{itemize}
\item $\ideal_\cB(\cA) \subset \Dom(\omega)$;  
\item for any $n \in \N, a_1,\dots, a_n \in \cA$ and any $b_1,\dots, b_n\in\cB$, we have that 
 \[
 \omega(a_1 b_1a_2b_2 \cdots a_n b_n) =\omega(a_1a_2\cdots a_n) \tau(b_1)\tau(b_2)\cdots \tau(b_n).    
 \]
\end{itemize}
\item Given $a_1,\dots, a_k \in \Dom(\omega)$ and $b_1,\dots, b_\ell \in \cC$, the pair 
$(\{a_1,\dots, a_k\},\{b_1,\dots, b_\ell\})$ is {\it cyclically monotone} if $(\alg\{a_1,\dots, a_k\}, \alg\{1_\cC,b_1,\dots, b_\ell\})$ 
is cyclically monotone.  Note that we {\it do not} assume that $\alg\{a_1,\dots, a_k\}$ contains the unit of $\cC$. 
\end{enumerate}
\end{definition}
\begin{remark} This definition is similar to monotone independence of Muraki \cite{Mur2001}, but the difference is 
 \[
 \omega(b_0 a_1 b_1a_2b_2 \cdots a_n b_n) =\omega(a_1a_2\cdots a_n) \tau(b_1)\tau(b_2)\cdots \tau(b_{n-1})\tau(b_0 b_n).     
 \]
The factor $\tau(b_0 b_n)$ is to be replaced by  $\tau(b_0)\tau(b_n)$ in the monotone case. 
\end{remark} 
\begin{definition}[Cyclic monotone product]  
Let $(\cA, \omega)$ be a non-commutative measure space and let $(\cB, \tau)$ be a non-commutative probability space. 
Let $\cA \ast \cB$ be the algebraic free product of $\cA$ and $\cB$,  $1_{\cB}$ being identified with the unit $\C1_{\cA\ast \cB} $, 
\[
\C1_{\cA\ast \cB} \oplus \cB^0\oplus \cA \oplus (\cA\otimes \cB^0) \oplus (\cB^0 \otimes \cA) \oplus  (\cA\otimes \cB^0 \otimes \cA) \oplus \cdots, 
\]  
where $\cB= \C1_{\cB} \oplus \cB^0$ is a direct sum decomposition as a vector space. 
We define a linear functional $\omega \tri\tau$ on an ideal of $\cA\ast\cB$ by 
 \begin{align}
& (\omega\tri\tau)(b_0 a_1 b_1a_2b_2 \cdots a_n b_n) :=\omega(a_1a_2\cdots a_n) \tau(b_1)\tau(b_2)\cdots \tau(b_{n-1})\tau(b_0 b_n),\\
&\Dom(\omega\tri\tau) := \Dom(\omega) \oplus (\Dom(\omega)\otimes \cB^0) \oplus (\cB^0 \otimes \Dom(\omega)) \oplus  (\Dom(\omega)\otimes \cB^0 \otimes \Dom(\omega)) \oplus \cdots 
\end{align}
for  $n\in\N, a_1,\dots, a_n \in \Dom(\omega)$, $b_0,b_1,\dots, b_n\in\cB\simeq \cB^0\oplus \C1_{\cA\ast \cB}$. 
We call $\omega\tri\tau$ the \emph{cyclic monotone product} of $\omega$ and $\tau$. 
\end{definition}

\begin{proposition}
The cyclic monotone product $\omega \tri \tau$ is a tracial positive linear functional on $\Dom(\omega\tri\tau)$ and 
hence $(\cA\ast \cB, \omega\tri\tau)$ becomes a non-commutative measure space.  
\end{proposition}

\begin{proof}
Let $x= b_0 a_1 b_1 \cdots a_n b_n$ and $y= b_0' a_1' b_1' \cdots a_m ' b_m'$ where $a_i, a_i' \in \Dom(\omega), b_i, b_i' \in \cB^0\cup \C1_{\cA\ast\cB}$. 
The traciality follows from 
\be
\begin{split}
(\omega \tri\tau)(x y) 
&= \omega(a_1 a_2 \cdots a_n a_1' \cdots a_m') \tau(b_0b_m') \tau(b_n b_0') \prod_{i=1}^{n-1} \tau(b_i)\prod_{i=1}^{m-1} \tau(b_i') \\
&= \omega(a_1' \cdots a_m' a_1 a_2 \cdots a_n) \tau(b_m'b_0) \tau(b_0' b_n) \prod_{i=1}^{n-1} \tau(b_i)\prod_{i=1}^{m-1} \tau(b_i') \\
&= (\omega \tri\tau)(y x).  
\end{split}
\ee
For the positivity it is enough to prove $(\omega\tri\tau)(x^* x)\geq 0$ when 
\be
x= \sum_{i=1}^{n}\lambda_{i} b_{i,0} a_{i,1}b_{i,1}a_{i,2}b_{i,2}\cdots a_{i, m(i)}b_{i,m(i)}, 
\ee
where $\lambda_{i}\in\C$, $a_{i, j}\in\Dom(\omega)$ and $b_{i, j}\in\cB^0 \cup\C1_{\cA\ast\cB}$. By the positivity of $\omega$ we have that 
\begin{align}
\omega\left(\left(\sum_{i=1}^{n}\lambda_{i} a_{i,1}\cdots a_{i,m(i)}\right)^*
\left(\sum_{i=1}^{n}\lambda_{i} a_{i,1}\cdots a_{i,m(i)}\right)\right)\geq 0. 
\end{align}
Then we have
\begin{align}
\sum_{i,j=1}^{n}\overline{\lambda_{i}}\lambda_j \omega(a_{i,m(i)}^*\cdots a_{i,1}^* a_{j,1}\cdots a_{j,m(j)})\geq0
\end{align}
for any $(\lambda_1, \dots, \lambda_n) \in \C^n$. 
Thus the $n\times n$ matrix 
\be
A=\left( \omega(a_{i,m(i)}^*\cdots a_{i,1}^* a_{j,1}\cdots a_{j,m(j)})\right)_{i,j=1}^n
\ee
 is positive definite. We also show that the matrices 
\begin{align}
&B':=\left(\tau(b_{i,0}^* b_{j,0}) \right)_{i,j=1}^n, \\
&B'':=\left(\tau(b_{i,1}^*)\tau(b_{j,1}) \cdots \tau(b_{i,m(i)-1}^*) \tau(b_{j,m(j)-1}) \tau(b_{i,m(i)}^*b_{j,m(j)})\right)_{i,j=1}^n 
\end{align}
are positive definite. It is easy to show that $B'$ is positive definite. The matrix $B''$ is also positive definite since 
\be
\begin{split}
&\sum_{i,j=1}^{n}\overline{\lambda_{i}}\lambda_{j}\tau(b_{i,1}^*)\tau(b_{j,1}) \cdots \tau(b_{i,m(i)-1}^*) \tau(b_{j,m(j)-1}) \tau(b_{i,m(i)}^*b_{j,m(j)}) \\
&= \tau\left(\left( \sum_{i=1}^n \lambda_i  \tau(b_{i,1})  \cdots \tau(b_{i,m(i)-1})  b_{i,m(i)}  \right)^* \left( \sum_{i=1}^n \lambda_i  \tau(b_{i,1})  \cdots \tau(b_{i,m(i)-1})  b_{i,m(i)}  \right)   \right) \geq0 
\end{split}
\ee
for any vector $(\lambda_1,\dots, \lambda_n) \in \C^n.$ 

The Schur product $A\circ B' \circ B''$ is also positive definite.  
Now using the definition of cyclic monotone product, we have 
\begin{align*}
0&\leq (\lambda_{1},\dots,\lambda_{n})^* ((A\circ B' \circ B'')(\lambda_{1},\dots,\lambda_{n}))&\\
&= (\omega\tri\tau)\left(\left(\sum_{i=1}^{n}\lambda_{i} b_{i,0}a_{i,1}b_{i,1}\cdots a_{i,m(i)}b_{i,m(i)}\right)^*\left(\sum_{i=1}^{n}\lambda_{i} b_{i,0}a_{i,1}b_{i,1}\cdots a_{i,m(i)}b_{i,m(i)}\right)\right). 
\end{align*}
Therefore we obtain $(\omega\tri\tau)(x^* x)\geq 0$.
\end{proof}

\begin{remark}\label{rem:not-faithful}
It turns out that this scalar product is not faithful at the algebraic level. 
For example, with the above notation, if $a,b$ are selfadjoint non-zero, and $\tau (b)=0$, then, calling $c=aba$, one
sees that $(\omega \tri\tau) (c^2)= \omega(a^4)\tau(b)^2=0$.
\end{remark}

Let $\widetilde{\tau}\colon\cA\ast\cB \to \C$ be the free product map of the zero map on $\cA$ and the trace $\tau$ 
on $\cB$. This map $\widetilde{\tau}$ is a tracial state on $\cA\ast\cB$, and gives a triple 
$(\cA\ast\cB, \omega \tri\tau, \widetilde{\tau} )$. Then $(\Dom(\omega), \cB)$ is cyclically monotone with respect 
to $( \omega\tri\tau, \widetilde{\tau})$.  This construction is a universal one.

When applying to random matrices we need the asymptotic version of independence. 

\begin{definition}
Let $(\cC_n, \omega_n, \tau_n)$ be non-commutative probability spaces with tracial weights. Let $a_1(n), \dots, a_k(n) \in \Dom(\omega_n)$ and $b_1(n), \dots, b_\ell(n)\in\cC_n$. The pair $(\{a_i(n)\}_{i=1}^k, \{b_i(n)\}_{i=1}^\ell)$ is {\it asymptotically cyclically monotone} if there exist a non-commutative probability space $(\cC, \omega, \tau )$ with a tracial weight, $a_1,\dots, a_k \in \Dom(\omega)$ and $b_1,\dots, b_\ell \in \cC$ such that 
\begin{enumerate}[\rm(1)]
\item $(\{a_1,\dots, a_k\}, \{b_1,\dots, b_\ell\})$ is cyclically monotone; 
\item for any non-commutative $\ast$-polynomial $P(x_1,\dots, x_k, y_1,\dots, y_\ell)$ such that 
$$P(0,\dots, 0,y_1,\dots, y_\ell)=0,$$ we have $P(a_1(n), \cdots, a_k(n), b_1(n),\dots, b_\ell(n)) \in \Dom(\omega_n)$ and 
\begin{align*}
\lim_{n\to\infty}\omega_n(P(a_1(n), \cdots, a_k(n), b_1(n),\dots, b_\ell(n))) = \omega(P(a_1,\dots, a_k, b_1,\dots, b_\ell)).  
\end{align*}
\end{enumerate}
\end{definition}

\subsection{Relation to infinitesimal freeness} Biane, Goodman and Nica introduced a kind of freeness related to 
type B noncrossing partitions in \cite{BGN2003} and then Belinschi and Shlyakhtenko formulated it as infinitesimal freeness, 
which is freeness with respect to a parametrized state $\tau_t ~(t\in (-\ep, \ep))$ up to the order $o(t)$ \cite{BeSh2012} 
(see also F\'evrier and Nica's work \cite{FN2010}). More precisely, let $\{\tau_t\}_{t \in (-\ep,\ep)}$ be a family of traces on 
an unital $\ast$-algebra $\cC$. Let $\tau:=\tau_0$ and suppose that $\tau_t$ is differentiable in the sense that the limit   
\be
\tau'(x):= \lim_{t\to0}\frac{\tau_t(x)- \tau(x)}{t}
\ee
exists for every $x\in\cC$. Moreover, suppose that $\ast$-subalgebras $\cA, \cB$ are free with respect to $\tau_t$ for every $t\in(-\ep,\ep)$, then we have a computation formula such as 
\begin{align}
\tau_t(a b) &= \tau_t(a)\tau_t(b), \\
\tau_t(a b a')&= \tau_t(a a ')\tau_t(b), \\
\tau_t(ab a' b')&= \tau_t(a a')\tau_t(b)\tau_t(b') + \tau_t(a)\tau_t(a')\tau_t(b b') - \tau_t(a)\tau_t( a')\tau_t(b)\tau_t(b'),   
\end{align}
for $a,a' \in \cA$ and $b,b' \in \cB$. Putting $t=0$ we get these formulas for $\tau$, and moreover if we take the derivative regarding $t$ at 0 we get the formulas 
\begin{align}
\tau'(a b) &= \tau'(a)\tau(b) + \tau(a)\tau'(b), \\
\tau'(a b a') &= \tau'(a a ')\tau(b) + \tau(a a')\tau'(b), \\
 \tau'(ab a' b') &= \tau'(a a') \tau(b)\tau(b') + \tau(a a') \tau'(b)\tau(b') +\tau'(a a') \tau(b)\tau'(b') \\
&+  \tau'(a)\tau(a')\tau(b b') + \tau(a)\tau'(a')\tau(b b')+ \tau(a)\tau(a')\tau'(b b') \notag\\
&- (4 \text{~terms}).  \notag
\end{align}
These computation formulas give some kind of universal formulas for mixed moments. 

Conversely, given $\tau, \tau'$ we can define infinitesimal freeness. 
 Let $\cC$ be a unital $\ast$-algebra, let $\tau$ be a tracial state on $\cC$ and let $\tau'$ be a tracial linear functional which 
 satisfies $\tau'(1_\cC)=0$. The triple $(\cC,\tau',\tau)$ is called an {\it infinitesimal non-commutative probability space} or 
 non-commutative probability space of type B. 
\begin{definition}\label{defIF} Let $(\cC,\tau',\tau)$ be an infinitesimal non-commutative probability space. 
Let  $\cA$ and $\cB$ be $\ast$-subalgebras of $\cC$, which may not contain the unit of $\cC$. Let $\tau_t:= \tau+t \tau'$ for 
$t \in \R$. We say that $\cA$ and $\cB$ are infinitesimally free if for any $n\in\N$, $a_1, \dots, a_n \in\cA$ and $b_1,\dots, b_n\in\cB$ 
\be
\tau_t((a_1 - \tau_t(a_1)1_\cC)(b_1-\tau_t(b_1)1_\cC)(a_2-\tau_t(a_2)1_\cC) \cdots (b_n-\tau_t(b_n)1_\cC)) = o(t). 
\ee
\end{definition}
\begin{remark} In the above definition $\tau_t$ may not be positive, but it does not matter in defining infinitesimal freeness. 
\end{remark}

Our cyclic monotone independence is a special case of infinitesimal freeness, which was essentially proved in \cite{Shl}. 
\begin{proposition}\label{PropIF}
Let $(\cC,\tau',\tau)$ be an infinitesimal non-commutative probability space and let $\cA,\cB$ be its $\ast$-subalgebras 
such that $1_\cC\in\cB$ and $\ideal_\cB(\cA) \subset \ker(\tau)$. 
Then  $\cA,\cB$ are infinitesimally free if and only if $(\cA,\cB)$ is cyclically monotone with respect to $(\tau',\tau)$. 
\end{proposition}
\begin{remark}
We assumed the positivity of the weight and the state in the definition of cyclic monotone independence, but we may drop it.   
The conclusion of Proposition \ref{PropIF} is to be understood in this generalized setup. 
\end{remark}
\begin{proof} Suppose that $\cA,\cB$ are infinitesimally free. 
Let $\tau_t:=\tau+t\tau'$. Since $\cA,\cB$ are free with respect to $\tau_t$ up to $o(t)$, it can be shown that  
\be\label{eq329}
\tau_t(a_1b_1\cdots a_n b_n)= \tau_t(a_1 a_2\cdots a_n) \tau_t(b_1) \tau_t(b_2)\cdots\tau_t(b_n) + R(t) +o(t), 
\ee
where $R(t)$ is the sum of monomials such that every monomial in $R(t)$ contains at least two factors of the form 
$\tau_t(a_{k_1}\cdots a_{k_p}), p <n$. For the proof apply the formula for products of free random 
variables \cite[Theorem 14.4]{NiSp2006}.  Our assumption implies that $\tau|_{\cA}=0$, and so $R(t)=O(t^2)$. 
Comparing the coefficients of $t$ in \eqref{eq329} we get the formula 
\be
\tau'(a_1b_1\cdots a_n b_n)= \tau'(a_1 a_2\cdots a_n) \tau(b_1) \tau(b_2)\cdots\tau(b_n), 
\ee
implying the cyclic monotone independence. 

For the converse direction, we check Definition \ref{defIF}. We have the decomposition  
\be
\begin{split}
&\tau_t((a_1 - \tau_t(a_1)1_\cC)(b_1-\tau_t(b_1)1_\cC)(a_2-\tau_t(a_2)1_\cC) \cdots (b_n-\tau_t(b_n)1_\cC)) \\
&= \tau((a_1 - \tau_t(a_1)1_\cC)(b_1-\tau_t(b_1)1_\cC)(a_2-\tau_t(a_2)1_\cC) \cdots (b_n-\tau_t(b_n)1_\cC)) \\
&\quad+ t \tau'((a_1 - \tau_t(a_1)1_\cC)(b_1-\tau_t(b_1)1_\cC)(a_2-\tau_t(a_2)1_\cC) \cdots (b_n-\tau_t(b_n)1_\cC)) \\
&=: J_1 + t J_2. 
\end{split}
\ee
By the assumption $\ideal_\cB(\cA) \subset \ker(\tau)$, we can see that 
\be
\begin{split}
J_1 
&= \tau((- \tau_t(a_1)1_\cC)(b_1-\tau_t(b_1)1_\cC)(-\tau_t(a_2)1_\cC) \cdots (b_n-\tau_t(b_n)1_\cC)) \\
&=(-1)^n t^n \tau'(a_1) \cdots \tau'(a_n) \tau((b_1-\tau_t(b_1)1_\cC)(b_2 -\tau_t(b_2)1_\cC) \cdots (b_n-\tau_t(b_n)1_\cC) )\\
&=O(t^2). 
\end{split}
\ee
Note that when $n=1$, we can show $J_1=O(t^2)$ since $ \tau(b_1-\tau_t(b_1)1_{\cC}) = t \tau'(b_1)$. For $J_2$, 
since $\tau_t(a_i)=-t \tau'(a_i)$ we can see that 
\be\label{eq330}
J_2 = \tau'(a_1 (b_1-\tau_t(b_1)1_\cC)a_2 \cdots a_n (b_n-\tau_t(b_n)1_\cC))+O(t). 
\ee
By cyclic monotonicity we obtain 
\be\label{eq331}
\begin{split}
 &\tau'(a_1 (b_1-\tau_t(b_1)1_\cC)a_2 \cdots a_n (b_n-\tau_t(b_n)1_\cC)) \\
 &=  \tau'(a_1 a_2 \cdots a_n) (\tau (b_1)-\tau_t(b_1)) \cdots (\tau(b_n)-\tau_t(b_n)) \\
 &= \tau'(a_1 a_2 \cdots a_n) (-t)^n \tau'(b_1)\tau'(b_2) \cdots \tau'(b_n) \\
&=O(t). 
\end{split}
\ee
Combining \eqref{eq330} and \eqref{eq331} we obtain $J_2=o(1)$. Therefore, $J_1+t J_2 = o(t)$. 
\end{proof}

For later use we define a notion of the convergence of elements in non-commutative probability spaces to elements in an 
infinitesimal non-commutative probability space. This definition is inspired by \cite{BeSh2012,Shl}.   
\begin{definition}\label{def convergenceIF}
Suppose that $(\cC,\tau',\tau)$ is an infinitesimal non-commutative probability space, $(\cC_n, \tau_n)$ are non-commutative 
probability spaces and $c_1,\dots, c_\ell \in \cC, c_1(n), \dots, c_\ell(n) \in\cC_n$. We say that the tuple 
$((c_1(n),\dots, c_\ell(n)), \tau_n)$ {\it converges up to the first order} to $((c_1,\dots, c_\ell), (\tau',\tau))$ if 
\begin{enumerate}[\rm(i)]
\item $((c_1(n),\dots, c_\ell(n)), \tau_n)$  converges in distribution to $((c_1,\dots, c_\ell),\tau)$, 
\item for any non-commutative $\ast$-polynomial $P(y_1,\dots, y_\ell)$,  
\[
\lim_{n\to\infty} \frac{\tau_n(P(c_1(n), \dots, c_\ell(n))) - \tau(P(c_1,\dots, c_\ell))}{1/n} =\tau'(P(c_1,\dots, c_\ell)). 
\]
\end{enumerate}
\end{definition}

\subsection{Eigenvalues of polynomials of cyclic monotone elements}\label{secEVCM}

In this section we compute the eigenvalues of polynomials of cyclic monotone elements. We need a combinatorial lemma 
to compute the eigenvalues of the anti-commutator. 

\begin{lemma}\label{lem1} For every integer $n\geq0 $ we have 
\[
\sum_{\ell =0}^{n-2 m} \binom{\ell +m-1}{m-1}\binom{n-\ell-m}{m} (-1)^\ell = \binom{[n/2]}{m}, \qquad m=0,1,\dots, [n/2], 
\]
where $\binom{p}{-1}:=\delta_{p,-1 }$ for integers $p$. 
\end{lemma}
\begin{proof} If $m=0$ the identity is easy, so we assume that $m \geq1$. 
The desired formula can be written in the form
\be
\sum_{p=0}^{\infty} \binom{p}{m-1}\binom{r-p}{m} (-1)^p = (-1)^{m-1}\binom{[(r+1)/2]}{m}, \qquad m=1,\dots, [(r+1)/2]. 
\ee
In order to prove this we use the identity 
\be
\frac{x^m}{(1-x)^{m+1}} = \sum_{p=0}^\infty \binom{p}{m}x^p. 
\ee
We observe that 
\be\label{eq92}
\begin{split}
\frac{(-x)^{m-1}}{(1+x)^m}\cdot \frac{x^m}{(1-x)^{m+1}}
&= \left(\sum_{p=0}^\infty \binom{p}{m-1}(-x)^p\right) \left(\sum_{q=0}^\infty \binom{q}{m}x^p\right) \\
&= \sum_{r=0}^\infty \left(\sum_{p=0}^\infty \binom{p}{m-1}\binom{r-p}{m} (-1)^p \right) x^r. 
\end{split}
\ee
The left hand side can be computed as 
\be\label{eq93}
\begin{split}
\frac{(-x)^{m-1}}{(1+x)^m}\cdot \frac{x^m}{(1-x)^{m+1}}
&= \frac{(-1)^{m-1}(1+x) x^{2m-1}}{(1-x^2)^{m+1}} \\
&= (-1)^{m-1}(1+x) \sum_{p=0}^\infty \binom{p}{m}x^{2 p-1} \\
&= (-1)^{m-1}\sum_{p=0}^\infty \binom{p}{m} (x^{2 p-1}+x^{2 p}) \\
&= (-1)^{m-1}\sum_{r=0}^\infty \binom{[(r+1)/2]}{m} x^r. 
\end{split}
\ee
The comparison of \eqref{eq92} and \eqref{eq93} implies the conclusion. 
\end{proof}

\begin{theorem}\label{compute} Let $(\cC,\omega, \tau)$ be a non-commutative probability space with a tracial weight. 
Let $a, a_1,\dots, a_k \in \Dom(\omega)$ and let $b, b_1,\dots, b_k \in \cC$. Suppose that $(a,a_1,\dots, a_k)$ has a trace 
class distribution with respect to $\omega$ (see Section \ref{sec3.1}) and that the pair $(\{a,a_1,\dots, a_k\}, \{b,b_1,\dots, b_k\})$ 
is cyclically monotone in $(\cC,\omega, \tau)$. 
\begin{enumerate}[\rm(1)] 
\item\label{A} Suppose that $a_1,\dots, a_k$ are selfadjoint. Let $B:=(\tau(b_i ^\ast b_j))_{i,j=1}^k \in M_k(\C)$. Then 
\[
\EV\left(\sum_{i=1}^k b_i a_i b_i^\ast\right) = \EV\!\left(\sqrt{B} \diag(a_1,\dots, a_k) \sqrt{B}\right),
\]
 where $\sqrt{B} \diag(a_1,\dots, a_k) \sqrt{B}$ is viewed as an element of 
 $(M_k(\C) \otimes \cC ,  \Tr_k \otimes \omega)$. 

\item\label{D} Suppose that $b_1,\dots, b_k$ are selfadjoint. Then 
\[
\EV\left(\sum_{i=1}^k a_i b_i a_i^\ast\right) = \EV\left(\sum_{i=1}^k \tau(b_i) a_ia_i^\ast\right). 
\]

\item\label{B} Suppose that $a,b$ are selfadjoint.  Let $p= \sqrt{\tau(b^2)} +\tau(b), q = -(\sqrt{\tau(b^2)} -\tau(b))$. Then 
$$
\EV(a b + b a) = (p \EV(a ) ) \sqcup (q \EV(a) ).
$$ 

\item\label{C} Suppose that $a,b$ are selfadjoint.   Let $r:= \sqrt{\tau(b^2) -\tau(b)^2}$. Then 
$$
\EV(\ri [a,b]) = (r \EV(a ) ) \sqcup (-r \EV(a) ).
$$
\end{enumerate}
\end{theorem}
\begin{remark} Our proofs are rather direct combinatorial arguments, but it may be possible to prove \eqref{B} and \eqref{C} 
by generalizing Nica and Speicher's computation of the distributions of commutators and anti-commutators of free random 
variables \cite{NiSp1998} to the infinitesimal free case.  
\end{remark}

\begin{proof}
\eqref{A} Let $x:=\sum_{i=1}^k b_i a_i b_i^\ast$. Since $B$ is symmetric nonnegative definite matrix, we can define $A = (\alpha_{i j})_{i,j=1}^k:= \sqrt{B}$. 
Then $\tau(b_i^\ast b_j)= \sum_{m} \alpha_{i m} \alpha_{m j}$. For $n\in \N$ we have 
\begin{equation}
\begin{split}
\omega(x^n) 
&= \sum_{i_1, \dots, i_n}\omega(b_{i_1} a_{i_1} b_{i_1}^\ast b_{i_2} a_{i_2} b_{i_2}^* \cdots b_{i_n} a_{i_n} b_{i_n}^*)  \\
&= \sum_{i_1, \dots, i_n}\tau(b_{i_1}^\ast b_{i_2}) \tau(b_{i_2}^\ast b_{i_3})\cdots \tau(b_{i_n}^\ast b_{i_1}) \omega(a_{i_1} \cdots a_{i_n}) \\
&=\sum_{ i_1, \dots, i_n, m_1,\dots,m_n}\alpha_{i_1 m_1} \alpha_{m_1 i_2} \alpha_{i_2 m_2} \alpha_{m_2 i_3}\cdots \alpha_{i_n m_n}\alpha_{m_n i_1}  \omega(a_{i_1} \cdots a_{i_n}) \\ 
&= \sum_{m_1,\dots,m_n} \omega\!\left(\left(\sum_{i_1} \alpha_{m_n i_1} a_{i_1} \alpha_{i_1 m_1} \right)\left(\sum_{i_2} \alpha_{m_1 i_2} a_{i_2} \alpha_{i_2 m_2} \right) \cdots \left(\sum_{i_n} \alpha_{m_{n-1} i_n} a_{i_n} \alpha_{i_n m_n} \right) \right) \\
&=  \sum_{m_1,\dots,m_n}\omega(\gamma_{m_n m_1}\gamma_{m_1 m_2}\cdots \gamma_{m_{n-1} m_n}) \\
&= (\Tr_k\otimes \omega)(G^n), 
\end{split}
\end{equation}
where $G =(\gamma_{\ell m})_{\ell,m=1}^k:= A \diag(a_1,\dots, a_k) A\in (M_k(\C) \otimes \cC)_{\rm sa}$. Corollary \ref{cor unique} implies that $\EV(x) = \EV(G)$.

\eqref{D} The claim follows from 
\begin{equation}
\begin{split}
\omega\left(\left(\sum_{i=1}^k a_i b_i a_i^\ast\right)^n\right) 
&= \sum_{i_1, \dots, i_n}\omega(a_{i_1}b_{i_1} a_{i_1}^\ast a_{i_2} b_{i_2} a_{i_2}^\ast \cdots a_{i_n} b_{i_n} a_{i_n}^\ast)  \\
&= \sum_{i_1, \dots, i_n}\omega(a_{i_1}\tau(b_{i_1}) a_{i_1}^\ast a_{i_2} \tau(b_{i_2}) a_{i_2}^\ast \cdots a_{i_n} \tau(b_{i_n}) a_{i_n}^\ast)  \\
&= \omega\left(\left(\sum_{i=1}^k \tau(b_i)a_i a_i^\ast\right)^n\right)
\end{split}
\end{equation}
and from Corollary \ref{cor unique}.

\eqref{B} We may assume that $\tau(b)\neq0$ since the general case can be covered by approximation.  
Let $c_1:= a b, c_2 := b a$. Then 
\be
\omega((a b+ b a)^n)=\sum_{(i_1, \dots, i_n)\in\{1,2\}^n} \omega(c_{i_1} \cdots c_{i_n}). 
\ee
Given $\bfi=(i_1,\dots, i_n)\in\{1,2\}^n$ the set of {\it ascents} is defined by 
\be
\up(\bfi):= \{1 \leq m \leq n \mid (i_m,i_{m+1}) =(1,2)\}
\ee
with the convention that $i_{n+1}=i_1$. Then 
\be
\omega(c_{i_1} \cdots c_{i_n}) = \omega(a^n) \tau(b)^n \alpha^{2 \# \up(\bfi)}, 
\ee
where $\alpha:=\frac{\tau(b^2)^{1/2}}{\tau(b)}$.
Let $\NC_{1,2;1}(n)$ be the set of noncrossing partitions of $\{1,\dots, n\}$ such that every block has cardinality $1$ (singleton) or 
$2$ (pair block), each singleton has depth $0$ or $1$ and each pair block has depth $0$.  
To each $\bfi =(i_1,\dots, i_n)\in\{1,2\}^n$ such that $i_1=1$ we associate $\pi(\bfi) \in \NC_{1,2; 1}(n)$ in the following procedure (with the convention that $i_{n+1}=i_1 ~(=1)$):   
\begin{itemize}
\item If $i_m =i_{m+1}$ and $1 \leq m \leq n$ then $\{m\}$ is a singleton of $\pi(\bfi)$; 

\item If $(i_m, i_{m+1}) = (1,2)$ and $1 \leq m \leq n-1$ then there exists a unique minimal integer $m+1 \leq \ell \leq n$ such that $(i_\ell, i_{\ell+1}) = (2,1)$. Then $\{m,\ell\}$ is a pair block of $\pi(\bfi)$. 
\end{itemize}
\begin{figure}
\setlength{\unitlength}{0.6mm}
\begin{minipage}{0.4\hsize}
\begin{center}
\begin{picture}(80,10)
          \put(0,0){\line(1,1){10}}
          \put(10,10){\line(1,0){20}}
          \put(30,10){\line(1,-1){10}}
          \put(40,0){\line(1,0){20}}
          \put(60,0){\line(1,1){10}}
            \put(70,10){\line(1,0){10}}
\end{picture}

\vspace{5mm}

$\bfi=(1,2,2,2,1,1,1,2,2)$

\begin{picture}(80,20)
          \put(0,0){\line(0,1){10}}
          \put(0,10){\line(1,0){30}}
         \put(10,0){\line(0,1){5}}
          \put(20,0){\line(0,1){5}}
          \put(30,0){\line(0,1){10}}
           \put(40,0){\line(0,1){10}}
            \put(50,0){\line(0,1){10}}
          \put(60,0){\line(0,1){10}}
           \put(60,10){\line(1,0){20}}
             \put(70,0){\line(0,1){5}}
            \put(80,0){\line(0,1){10}}
\end{picture}
\caption{$\pi(1,2,2,2,1,1,1,2,2)$}
\end{center}
\end{minipage}
\begin{minipage}{0.4\hsize}
\begin{center}
\begin{picture}(80,10)
          \put(0,0){\line(1,0){20}}
           \put(20,0){\line(1,1){10}}
            \put(30,10){\line(1,0){10}}
           \put(40,10){\line(1,-1){10}}
          \put(50,0){\line(1,1){10}}
          \put(60,10){\line(1,-1){10}}
           \put(70,0){\line(1,0){10}}
\end{picture}

\vspace{5mm}

$\bfi=(1,1,1,2,2,1,2,1,1)$

\begin{picture}(80,20)
           \put(0,0){\line(0,1){10}}
           \put(10,0){\line(0,1){10}}
           \put(20,10){\line(1,0){20}}
          \put(20,0){\line(0,1){10}}
           \put(40,0){\line(0,1){10}}
            \put(30,0){\line(0,1){5}}
          \put(50,0){\line(0,1){10}}
          \put(60,0){\line(0,1){10}}
          \put(50,10){\line(1,0){10}}
          \put(70,0){\line(0,1){10}}
            \put(80,0){\line(0,1){10}}
\end{picture}
\caption{$\pi(1,1,1,2,2,1,2,1,1)$}
\end{center}
\end{minipage}
\end{figure}
The correspondence $\{(i_1,\dots, i_n)\in\{1,2\}^n\mid i_1=1\}\to \NC_{1,2; 1}(n)$ is bijective and $\#\up(\bfi)= \#\pair(\pi(\bfi))$.  Then 
\be
\begin{split}
\sum_{\substack{(i_1, \dots, i_n)\in\{1,2\}^n\\ i_1=1}} \omega(c_{i_1} \cdots c_{i_n}) 
&=  \omega(a^n) \tau(b)^n \sum_{\pi \in \NC_{1,2;1}(n)}\alpha^{2 \# \pair(\pi)} \\
&=  \omega(a^n) \tau(b)^n\sum_{m=0}^{[n/2]}  \sum_{\substack{\pi \in \NC_{1,2; 1}(n)\\ \# \pair(\pi)=m}}\alpha^{2 m} \\
 &=  \omega(a^n) \tau(b)^n\sum_{m=0}^{[n/2]}  \binom{n}{2 m} \alpha^{2 m} \\
 &=  \omega(a^n) \tau(b)^n\frac{(1+\alpha)^n + (1-\alpha)^n}{2}. 
\end{split}
\ee
We then compute the sum over $\bfi$ such that $i_1=2$. Let $i \mapsto i^\star$ be the flip of $1$ and $2$. Then the map $\bfi \mapsto \bfi^\star := (i_1^\star, \dots, i_n^\star)$ defines an involution on $\{1,2\}^n$ whose restriction to $\{(i_1,\dots, i_n)\in\{1,2\}^n\mid i_1=1\}$ is a bijection onto $\{(i_1,\dots, i_n)\in\{1,2\}^n\mid i_1=2\}$. Moreover, it holds that 
\be
\omega(c_{i_1} \cdots c_{i_n}) = \omega(c_{i_1^\star} \cdots c_{i_n^\star}). 
\ee
Thus 
\be
\sum_{\substack{(i_1, \dots, i_n)\in\{1,2\}^n\\ i_1=2}} \omega(c_{i_1} \cdots c_{i_n})  = \sum_{\substack{(i_1, \dots, i_n)\in\{1,2\}^n\\ i_1=1}} \omega(c_{i_1} \cdots c_{i_n}). 
\ee
Therefore we conclude that 
\begin{align}\label{formula5c}
\begin{split}
\omega((a b+ b a)^n)
 &=  \omega(a^n) \tau(b)^n((1+\alpha)^n + (1-\alpha)^n) \\
 &= \omega(a^n) (p^n + q^n),   
\end{split}
\end{align}
so Corollary \ref{cor unique} implies the conclusion.

\eqref{C} Let $d_1:= a b, d_2 := - b a$ as before. Then 
\be
\omega((\ri[a,b])^n)=\sum_{(i_1, \dots, i_n)\in\{1,2\}^n} \ri^n \omega(d_{i_1} \cdots d_{i_n}). 
\ee
We show that $\omega((\ri[a,b])^n)=0$ when $n$ is odd. If we apply the involution $\bfi \mapsto \bfi^\star$ defined in the proof of \eqref{B}, then 
\be\label{involution}
\omega(d_{i_1}\cdots d_{i_n}) = (-1)^n \omega(d_{i_1^\star}\cdots d_{i_n^\star}). 
\ee
Therefore if $n$ is odd then 
\be
\begin{split}
\omega((\ri[a,b])^n)
&=\sum_{(i_1, \dots, i_n)\in\{1,2\}^n} \ri^n \omega(d_{i_1} \cdots d_{i_n})\\
&=\sum_{\substack{(i_1, \dots, i_n)\in\{1,2\}^n\\i_1=1}} \ri^n \omega(d_{i_1} \cdots d_{i_n})+\sum_{\substack{(i_1, \dots, i_n)\in\{1,2\}^n\\i_1=1}} \ri^n \omega(d_{i_1^\star} \cdots d_{i_n^\star})\\
&= 0. 
 \end{split}
\ee

Hereafter we may assume that $n$ is even. Let $\insing(\pi)$ be the set of inner singletons of $\pi \in \NC_{1,2;1}(n)$, that is, the set of singletons with depth 1. Given $\bfi=(i_1,\dots, i_n)\in\{1,2\}^n$, it holds that 
\be
 \omega(d_{i_1} \cdots d_{i_n}) = (-1)^{\#\insing(\pi(\bfi))+\#\pair(\pi(\bfi))} \omega(a^n)\tau(b)^n \alpha^{2\#\pair(\pi(\bfi))}, 
\ee
where $\alpha$ is the number defined in the proof of \eqref{B}. For fixed integers $0\leq m\leq n/2, 0 \leq \ell \leq n-2 m$, we have that 
\be
\#\{\pi\in\NC_{1,2;1}(n)\mid \#\pair(\pi)=m, \#\insing(\pi)=\ell\} = \binom{\ell+m-1}{m-1}\binom{n-m-\ell}{m}, 
\ee
which is to be understood as 1 when $m= \ell=0$ and 0 when $m=0, \ell >0$. This is because there are $\binom{\ell+m-1}{m-1}$ ways to place $\ell$ unlabeled singletons inside $m$ labeled pair blocks, and then there are $\binom{n-m-\ell}{m}$ ways to place the other $n-2 m-\ell$ unlabeled singletons outside the $m$ labeled pair blocks (the number of ways to distribute $n-2m-\ell$ unlabeled balls into $m+1$ labeled boxes). 
Thus 
\be
\begin{split}
&\sum_{\substack{(i_1, \dots, i_n)\in\{1,2\}^n\\i_1=1}} \ri^n \omega(d_{i_1} \cdots d_{i_n}) \\
&\qquad= \ri^n \omega(a^n)\tau(b)^n\sum_{m=0}^{n/2} (-\alpha^2)^m \sum_{\ell=0}^{n-2 m} \binom{\ell+m-1}{m-1}\binom{n-m-\ell}{m}(-1)^\ell, 
\end{split}
\ee
and by lemma \ref{lem1}, we have that 
\be
\begin{split}
\sum_{\substack{(i_1, \dots, i_n)\in\{1,2\}^n\\i_1=1}} \ri^n \omega(d_{i_1} \cdots d_{i_n}) 
&= \ri^n \omega(a^n)\tau(b)^n\sum_{m=0}^{n/2} (-\alpha^2)^m \binom{n/2}{m}\\
&= (-1)^{n/2} \omega(a^n)\tau(b)^n  (1-\alpha^2)^{n/2}\\
&= \omega(a^n) r^{n}. 
\end{split}
\ee
Since we know \eqref{involution}, we get 
\be
\omega((\ri[a,b])^n) = 2 r^n \omega(a^n)
\ee
for even $n$. If $\lambda_1,  \lambda_2,\dots$ are the eigenvalues of $a$ then for every nonnegative integer $n$ 
\begin{equation}\label{formula5d}
\omega( (\ri[a,b])^n ) = (r\lambda_1)^n + (r \lambda_2)^n +\cdots + (-r\lambda_1)^n + (-r \lambda_2)^n+\cdots, 
\end{equation}
showing the conclusion by Corollary \ref{cor unique}. 
\end{proof}

\section{The link with random matrices}\label{sec:rmt}
In this section, we shall introduce and prove asymptotic cyclic monotone independence of Haar invariant random matrices 
which have limiting compact distributions and random matrices which have limiting distributions with respect to the normalized 
trace. Theorem \ref{ACM} is essentially equivalent to \cite[Lemma 3.1, Lemma 3.2]{Shl} by Shlyakhtenko. 
We supply a proof of this theorem using the {\it Weingarten calculus}. Actually in our proof we can remove the assumption of some 
norm boundedness in Lemma 3.1 of Shlyakhtenko, and so we can unify the proofs of \cite[Lemma 3.1, Lemma 3.2]{Shl}. 
Moreover, we can prove the almost sure asymptotic independence and almost sure convergence of discrete eigenvalues, 
still without the norm boundedness in the trace class setting.  
We  generalize the results to the compact setup, but then we need the norm boundedness.

First, we introduce the tool called the Weingarten calculus,  summarizing results in \cite{CoSn2006} (see also the 
references \cite{Collins2003} and \cite[Lecture 23]{NiSp2006}). 
\subsection{The Weingarten calculus}
Let $S_{I}$ be the symmetric group acting on a finite set $I$ and in particular let $S_k$ be the symmetric group 
$S_{\{1,2,\dots, k\}}$. The identity of $S_{k}$ is denoted by $1_{k}$. 
Let $\bE$ be an expectation in a probability space and let $U=\left(u_{ij}\right)_{i,j=1}^{n}$ be a normalized Haar unitary 
random matrix, i.e.\ the law of $U$ is the normalized Haar measure on the unitary group $U_n$. Let $\cE$ be a linear 
map of $M_n(\C)^{\otimes k}$ defined by 
\be
\cE(A) =  \bE[U^{\otimes k} A (U^\ast)^{\otimes k}].  
\ee
Let $\{\delta_\sigma\}_{\sigma \in S_k}$ be the canonical basis of $\C[S_k]$ and let $\Phi$ be a linear map from $M_n(\C)^{\otimes k}$ to $\C[S_k]$ defined by 
\be
\Phi(A) = \sum_{\sigma\in S_k} \Tr_{M_n(\C)^{\otimes k}}( \rho(\sigma)^\ast A ) \delta_\sigma,  
\ee 
where $\rho: S_k \to M_n(\C)^{\otimes k}$ is the natural representation
\be
(\rho(\sigma))(v_1 \otimes \cdots \otimes v_k) := v_{\sigma^{-1}(1)}\otimes \cdots \otimes v_{\sigma^{-1}(k)}, \qquad v_i\in \C^n. 
\ee 
For $\sigma\in S_{k}$ and $n\geq k$, 
we define the Weingarten function $\Wg(\sigma,n)$ as
\be
\Wg(\sigma, n) = \bE[u_{11}\dots u_{k k}\overline{u}_{1\sigma(1)}\dots\overline{u}_{k\sigma(k)}]. 
\ee
It follows from \cite[Eq.\ (9)]{CoSn2006} that the function $\sigma\mapsto\Wg(\sigma,n), S_k \to \C$ is a linear combination 
of the characters of irreducible representations of $S_k$, and so $\Wg(\tau \sigma \tau^{-1}, n)=\Wg(\sigma,n)$. The conjugate 
classes of a symmetric group are determined by the structure of cycle decomposition. Since $\sigma$ and $\sigma^{-1}$ have 
the same structure of cycle decomposition, we have $\Wg(\sigma^{-1},n)=\Wg(\sigma,n)$ for every $\sigma \in S_k$.

What is important is asymptotics of the Weingarten function for large $n$. 
Let $\Cat_p$ be the Catalan number 
\be
\Cat_p:= \frac{(2 p)!}{p!(p+1)!},\qquad p \in\N.   
\ee
For $\sigma \in S_k$ let $|\sigma|$ be the length function, that is, the minimal number of transpositions to express $\sigma$ as the product of them. 
Let $\sigma = c_1 \cdots c_{\ell(\sigma)}$ be the cycle decomposition of $\sigma \in S_k$ and then let 
\be\label{Moeb}
\Moeb(\sigma):= \prod_{i=1}^{\ell(\sigma)}  (-1)^{|c_i|}\Cat_{|c_i|}.   
\ee
Note that $|\sigma|= \sum_{i=1}^{\ell(\sigma)} |c_i|$ and $\ell(\sigma) = k- |\sigma|$.  Then it is known that \cite[Corollary 2.7]{CoSn2006}
\be\label{eq:decay}
\Wg(\sigma,n) = n^{-k-|\sigma|}(\Moeb(\sigma)+O(n^{-2})). 
\ee

The Weingarten function can obviously be regarded as the element of $\C[S_k]$ 
\be
\Wg= \sum_{\sigma \in S_k} \Wg(\sigma, n)\delta_\sigma. 
\ee
It was shown in \cite{CoSn2006} that for all $A, B \in M_n(\C)^{\otimes k}$
\be\label{Wg1}
\Phi(A \cE(B)) = \Phi(A) \Phi(B) \Wg.  
\ee
In this paper we  consider $A=A_1\otimes \cdots \otimes A_k,B=B_1\otimes \cdots \otimes B_k$ where 
$A_i, B_i \in M_n(\C), i=1,\dots,k$. For a cycle $c=(i_1i_2 \dots i_m)$, let $A_c$ be the product 
$A_{i_1} \cdots A_{i_m}$ and let $\Tr_\sigma$ be the product of traces $\Tr_n$ according to the cycle 
decomposition $\sigma=c_1 c_2\cdots c_{\ell(\sigma)}$, 
\be
\Tr_\sigma(A_1,\dots, A_k)= \prod_{i=1}^{\ell(\sigma)} \Tr_n(A_{c_i}). 
\ee 
For example if $\sigma = (1 , 3) (2)$ then $\Tr_\sigma (A_1,A_2,A_3)= \Tr_n(A_1A_3)\Tr_n(A_2)$. 
Note that this notation is well-defined thanks to the cyclic property of the trace. 
We can then show that 
\be
\Tr_{M_n(\C)^{\otimes k}}( \rho(\sigma)^\ast A ) = \Tr_\sigma(A_1,\dots, A_k). 
\ee 
Then \eqref{Wg1} reads, for every $\sigma\in S_k$,  
\be\label{Wg2}
\begin{split}
&\bE[\Tr_\sigma\!\left(A_1 U B_1 U^*,  \cdots, A_k U B_k U^*\right)] \\
&\qquad= \sum_{\substack{\sigma_1,\sigma_2,\sigma_3 \in S_k \\ \sigma_1 \sigma_2\sigma_3 =\sigma}} \Tr_{\sigma_1}(A_1,\dots, A_k) \Tr_{\sigma_2}(B_1,\dots, B_k) \Wg(\sigma_3,n). 
\end{split}
\ee
This formula is the main tool of our analysis below. 

\subsection{Asymptotic cyclic monotone independence of random matrices} 
We prove asymptotic cyclic monotone independence on average in the trace class setup. Note that we can 
replace $(\{A_1, \dots,A_k\}, \{UB_1U^*, \dots, UB_\ell U^*\})$  with $(\{U A_1 U^*, \dots, U A_kU^*\}, \{B_1, \dots, B_\ell\})$ 
in the following asymptotic results. 
\begin{theorem} \label{ACM}
Let $U=U(n)$ be an $n\times n$ Haar unitary and $A_i=A_i(n), B_j=B_j(n), i=1,\dots,k$, $j=1,\dots, \ell$ be $n\times n$ random 
matrices. Suppose that 
\begin{enumerate}[\rm(1)]
\item\label{A11} $((A_1,\dots, A_k), \bE\otimes \Tr_n)$ converges in distribution to a $k$-tuple of trace class operators as $n\to\infty$,  
\item\label{B11} $((B_1,\dots, B_\ell), \bE\otimes \tr_n)$ converges in distribution to an $\ell$-tuple of elements in a non-commutative probability space as $n\to\infty$,  
\item $\{A_1,\ldots ,A_k\}, \{B_1,\ldots ,B_\ell\}, U$ are independent. 
\end{enumerate}
Then the pair $(\{A_i\}_{i=1}^k, \{U B_jU^*\}_{j=1}^\ell)$ is asymptotically cyclically monotone with respect to $(\bE\otimes \Tr_n,\bE \otimes \tr_n)$. 
\end{theorem} 
\begin{proof}
We may assume that $\ell=k$ since otherwise we may add $0$'s to $A_i$'s or the identity matrices to $B_j$'s. 
For simplicity $U B_i U^*$ is abbreviated to $B_i$. Note that it suffices to show that the expectation
\be\label{eq222}
\bE[\Tr_n(A_1B_1\ldots A_kB_k)] 
\ee
factorizes following cyclic monotone independence, since a general monomial 
can be written in the form $B_0' A_1' \cdots A_m' B_{m+1}'$ where 
$A_i' \in \alg\{A_1, \cdots, A_k\}$ and $B_i' \in \alg\{1_n, B_1,\dots, B_k\}$. By the traciality, the distribution of 
$B_0' A_1'B_1' \cdots A_m' B_{m+1}'$ with respect to the non-normalized trace is the same as that of 
$A_1'B_1' \cdots A_m' B_{m+1}' B_0'$, so every monomial reduces to \eqref{eq222}. 
Also, a concrete construction of a limiting non-commutative probability space with a tracial weight can be given by the 
cyclic monotone product in Definition \ref{construction}. Thus, it suffices to show the factorization of the expectation \eqref{eq222}.

Let $Z$ be the circular permutation, i.e.\ $Z=(1\ldots k)$. In this context, by the Weingarten formula \eqref{Wg2} we have
\be
\begin{split}
&\bE[\Tr_n(A_1B_1\cdots A_kB_k)]\\
&\qquad=
\sum_{\substack{\sigma_1,\sigma_2,\sigma_3\in S_k\\ \sigma_1\sigma_2\sigma_3=Z}}\bE[\Tr_{\sigma_1}(A_1,\dots, A_k)]\bE[\Tr_{\sigma_2}(B_1,\dots, B_k)] \Wg(\sigma_3,n). 
\end{split}
\ee
Next we make a decay analysis. 
Since $(A_1,\dots, A_k)$ has a limiting trace class distribution, then the leading behavior of 
$\Tr_{\sigma_1}(A_1,\dots,A_k)$ is $O(n^0)$.
Since $(B_1,\dots, B_k)$ has a limiting distribution with respect to the normalized trace,  the behavior of 
$\Tr_{\sigma_2}(B_1,\dots, B_k)$ is $O(n^{k-|\sigma_2|})$.
Finally, we know that $\Wg(\sigma_3,n)$ behaves as $O(n^{-k-|\sigma_3|})$. Therefore, for a triple $\sigma_1,\sigma_2,\sigma_3$ such that 
$\sigma_1\sigma_2\sigma_3=Z$, the contribution of the summand is
$O(n^{-|\sigma_2|-|\sigma_3|})$.
Therefore, the asymptotics is driven by the summands for which $|\sigma_2|=|\sigma_3|=0$ i.e. $\sigma_2=\sigma_3=1_k$. This forces $\sigma_1=Z$, and there is only one such summand. As a conclusion, 
\be
\bE[\Tr_n(A_1B_1\cdots A_kB_k)]= \bE[\Tr_n(A_1 A_2 \cdots A_k)] \prod_{i=1}^k \bE[\tr_n(B_i)]+O(n^{-1}).
\ee
\end{proof}

\subsection{Almost sure convergence}

Some further analysis shows that
$$\mathrm{Cov}(\Tr (A_1B_1\ldots A_kB_k))=O(n^{-1}).$$
Interestingly, this is quite different from classical random matrix models inspired from free probability theory
\cite{Collins2003, CoSn2006}, where the covariance behaves rather like $O(n^{-2})$. 
As it is classically known, a behavior of $O(n^{-2})$, summable in $n$, allows to prove the almost sure convergence of the 
traces of random matrix models. 
Here, since we do not have the behavior $O(n^{-2})$, we need to investigate other functionals. 

From now on, we assume that the matrices $A_i, B_i$ are deterministic, but by conditioning $A_i,B_i$ to be constant, we can 
generalize the results to the case when $A_i, B_i$ are random matrices independent of the Haar unitary $U$ such 
that $((A_1,\dots, A_k), \Tr_n)$ and $((B_1,\dots, B_k), \tr_n)$ almost surely converge in distributions to deterministic elements. 
For example this allows us to take $B_1= \cdots = B_\ell$ to be a GUE $G$, and in this case we do not need to rotate 
$G$ by a Haar unitary since $G$ is rotationally invariant.

\begin{lemma}\label{ACM2} 
Let $A_i=A_i(n), B_i=B_i(n), i=1,\dots, k,$ be $n\times n$ deterministic matrices and let $U= U(n)$ be a Haar unitary 
random matrix. Suppose that 
\begin{enumerate}[\rm(1)]
\item\label{A12} $((A_1,\dots, A_k), \Tr_n)$ converges in distribution to a $k$-tuple of trace class operators as $n\to\infty$,  
\item\label{B12} $((B_1,\dots, B_k), \tr_n)$ converges in distribution to a $k$-tuple of elements in a non-commutative probability 
space as $n\to\infty$. 
\end{enumerate}
Then 
$$
\bE\!\left[ \left| \Tr_n (A_1 U B_1 U^* \cdots A_k UB_kU^*)-\bE[\Tr_n (A_1 UB_1U^*\cdots A_k UB_kU^*)] \right|^4\right]=O(n^{-2}). 
$$
\end{lemma}

\begin{proof}
We prove a more general result, namely, instead of taking $k$ matrices $A_i$ and $k$ matrices
$B_i$, we take $4k$ matrices for each kind satisfying the same assumption. Moreover, for notational simplicity $U B_i U^*$ is 
abbreviated to $B_i$. Let 
\begin{align}
&X_i=  \Tr_n\! \left(A_{(i-1)k+1} B_{(i-1)k+1} \cdots A_{(i-1)k +k} B_{(i-1)k+k}\right), \qquad i=1,2,3,4,  \\
&\mathring{X}_i= X_i - \bE[X_i], 
\end{align}
which are complex-valued random variables. We prove that 
\be\label{eq:power4}
\begin{split}
& \bE[ \mathring{X}_1 \mathring{X}_2 \mathring{X}_3 \mathring{X}_4]=O(n^{-2}). 
\end{split}
\ee 
This implies the lemma if we define $X_2,X_3,X_4$ such that $X_3=X_1, X_4 = X_2=\overline{X_1}$. The last condition $X_2=\overline{X_1}$ can be realized by taking 
$A_{k+1}= A_k^*, B_{k+1}= B_{k-1}^*, \dots, A_{2k}=A_1^*, B_{2k}=B_k^*$,  since 
\be
\begin{split}
\overline{X_1} 
&= \Tr_n\!\left((A_{1} B_{1} \cdots A_{k} B_{k})^*\right)= \Tr_n(A_{k}^* B_{k-1}^* A_{k-1}^*  \cdots A_2^* B_1^* A_1^* B_{k}^*). 
\end{split}
\ee

We shall denote by $I_i$ the interval $\{(i-1)k+1 ,(i-1)k +2 ,\ldots ,(i-1)k+k\},$ by $Z_i$ the cyclic 
permutation $((i-1)k+1 ,(i-1)k +2 ,\ldots ,(i-1)k+k)$ of $S_{I_i}$ and by $Z^{\cup 4}$ the permutation $Z_1Z_2Z_3Z_4$ of $S_{4 k}$. 
Thanks to the Weingarten formula \eqref{Wg2} we have formulas to compute the moments 
$\bE[ \mathring{X}_1 \mathring{X}_2 \mathring{X}_3 \mathring{X}_4]$. 
If we expand this mixed moment, we obtain $16$ (products of) expectations. 
Our notation is 
\begin{align}
\bE[ \mathring{X}_1 \mathring{X}_2 \mathring{X}_3 \mathring{X}_4]=\sum_{A \subset \{1,2,3,4\} }\bE_A, \\
\bE_A= (-1)^{\# A} \bE\!\left[\prod_{i\in A} X_i\right] \prod_{i \in A^c}\bE[ X_i]. 
\end{align}
For example, our notation means $\bE_{\{1,3,4\}} = -\bE[X_1 \bE[X_2] X_3 X_4]$. 
In this specific example, the Weingarten formula 
boils down to a sum over permutations in $S_{I_1 \cup I_3 \cup I_4}$ for the evaluation of  $\bE [X_1X_3X_4]$, 
and those in $S_{I_2}$ for the evaluation of $\bE [X_2]$.

It follows from the Weingarten formula that for each $A\subset \{1,2,3,4\}$ there exists a
number  $f_A (\sigma_1,\sigma_2,n)$ such that
\be
\bE_A=\sum_{\substack{\sigma_1,\sigma_2,\sigma_3\in S_{4k} \\
\sigma_1\sigma_2\sigma_3=Z^{\cup 4}}}\Tr_{\sigma_1}(A_1,\dots, A_{4k})\Tr_{\sigma_2}(B_1,\dots, B_{4k})f_A (\sigma_1,\sigma_2,n).
\ee
Note that $\sigma_3$ may be omitted but it is written for clearer understanding. The number $f_A$ is either zero, or a product 
of (two, three or four) Weingarten functions. Again, instead of going through heavy notation that would perhaps not ease 
the understanding, let us describe $f_A (\sigma_1,\sigma_2,n)$ in the generic case $A=\{1,3,4\}$. Recall that 
$\bE_A = -\bE[X_2]\bE[X_1 X_3 X_4]$. The product of Weingarten formulas \eqref{Wg2} for $\sigma= Z_2$ and for 
$\sigma =Z_1 Z_3 Z_4$ only gives  permutations of $S_{I_2} \times S_{I_1 \cup I_3 \cup I_4}$, and so, for 
$\sigma_1,\sigma_2,\sigma_3 \in S_{4 k}$ such that $\sigma_1\sigma_2\sigma_3=Z^{\cup 4}$, 
\be
f_A (\sigma_1,\sigma_2,n)
=
\begin{cases}
- \Wg (\sigma_{3| I_2},n)\Wg (\sigma_{3| I_1 \cup I_3 \cup I_4},n), \\ \qquad\qquad \text{if $\sigma _1, \sigma_2$ leave 
$I_2$ and $I_1 \cup I_3 \cup I_4$ invariant},  \\
0, \qquad\quad \text{otherwise}. 
\end{cases}
\ee
The general case is alike, and hence  we can define a function $f_A (\sigma_1,\sigma_2,n)$ for every subset $A$. 
This gives us the expression 
\be
\begin{split}
&\bE[ \mathring{X}_1 \mathring{X}_2 \mathring{X}_3 \mathring{X}_4]
=\sum_{\substack{\sigma_1,\sigma_2,\sigma_3\in S_{4k} \\
\sigma_1\sigma_2\sigma_3=Z^{\cup 4}}}\Tr_{\sigma_1}(A_1,\dots, A_{4k})\Tr_{\sigma_2}(B_1,\dots, B_{4k}) f (\sigma_1,\sigma_2,n),
\end{split}
\ee 
where 
\be\label{coefficients}
f (\sigma_1,\sigma_2,n)= \sum_{A\subset \{1,2,3,4\} }f_A (\sigma_1,\sigma_2,n).
\ee
As seen above, $f (\sigma_1,\sigma_2,n)$ is a signed sum of products of Wg functions on various permutation groups.

By assumption, $\Tr_{\sigma_1}(A_1,\dots, A_{4k})=O(1)$ as $n\to\infty$ and $\Tr_{\sigma_2}(B_1,\dots, B_{4k})= O(n^{\ell(\sigma_2)})=O(n^{4k-|\sigma_2|})$. From \eqref{eq:decay} it follows that 
\be
\Tr_{\sigma_2}(B_1,\dots, B_{4k}) f (\sigma_1,\sigma_2,n)=O(n^{-|\sigma_2|-|\sigma_3|}).
\ee  
We show that actually
\be\label{eq:decay2}
\Tr_{\sigma_2}(B_1,\dots, B_{4k}) f (\sigma_1,\sigma_2,n)=O(n^{-2}),  
\ee 
or equivalently 
\be
f (\sigma_1,\sigma_2,n)=O(n^{-4 k + |\sigma_2|-2}) 
\ee
for each fixed $\sigma_1, \sigma_2~(\text{and}~\sigma_3) \in S_{4k}$ (such that $\sigma_1\sigma_2\sigma_3=Z^{\cup4}$). 

Given $\sigma_1,\sigma_2,\sigma_3 \in S_{4 k}$ such that $\sigma_1\sigma_2\sigma_3 =Z^{\cup 4}$, we define an equivalence 
relation on $\{1,\dots, 4k\}$: $i \sim j$ if there exists $ \tau \in {\rm Grp}\langle \sigma_1,\sigma_2,\sigma_3\rangle$ such that 
$\tau(i)=j$. Since this group contains $Z^{\cup 4}$, every interval $I_i$ must be a subset of some equivalence class. 
Then, the permutations $\sigma_1,\sigma_2,\sigma_3$ associate a set partition 
$\pi(\sigma_1,\sigma_2,\sigma_3)=\{P_1,\dots, P_m\}$ of $\{1,2,3,4\}$ such that the subsets of $\{1,\dots, 4 k\}$, 
\be
\bigcup_{i \in P_1} I_i,\bigcup_{i \in P_2} I_i,\dots, 
\ee
are exactly the equivalence classes generated by the actions of $\sigma_1,\sigma_2,\sigma_3$. A subset $A$ of 
$\{1,2,3,4\}$ also associates the set partition $\pi(A)=\{A, \{b\}: b \in A^c\}$ of $\{1,2,3,4\}$. Then only $f_A$ for 
which $\pi(A)$ is coarser than or equal to $\pi(\sigma_1,\sigma_2,\sigma_3)$ contributes to $f$ in the 
sum \eqref{coefficients} and the other $f_A$'s are zero.  We distinguish several cases according to the set partition 
associated to $\sigma_1,\sigma_2,\sigma_3$. 

\begin{enumerate}[\rm(i)]
\item  $\pi(\sigma_1,\sigma_2,\sigma_3)=\{\{1,2,3,4\}\}$, or equivalently, the group 
${\rm Grp}\langle \sigma_1,\sigma_2,\sigma_3\rangle$ acts on $\{1,2,\dots,4k\}$ transitively. In this case, $f_A$ is zero unless $A=\{1,2,3,4\}$, and so $f=f_{\{1,2,3,4\}}=\Wg(\sigma_3,n)$. We have to exclude $|\sigma_2|+|\sigma_3|=0$, as the condition 
$\sigma_1\sigma_2\sigma_3=Z^{\cup 4}$ contradicts transitivity. We can also exclude $|\sigma_2|+|\sigma_3|=1$. Indeed this means that
one is the identity, and the other one is a transposition, and again the condition 
$\sigma_1\sigma_2\sigma_3=Z^{\cup 4}$ is incompatible with a transitive action. 
Therefore, $|\sigma_2|+|\sigma_3|\geq 2$, and so $\Tr_{\sigma_2}(B_1,\dots, B_{4k})f (\sigma_1,\sigma_2,n)=O(n^{-2})$.

\item $\pi(\sigma_1,\sigma_2,\sigma_3)$ is a pair partition, e.g.\ $\{\{1,2\}, \{3,4\}\}$. In this case again we have 
$
f= f_{\{1,2,3,4\}} = \Wg(\sigma_3,n),  
$
and from a similar reasoning we must have $|\sigma_2|+|\sigma_3| \geq2$ and hence 
$\Tr_{\sigma_2}(B_1,\dots, B_{4k})f (\sigma_1,\sigma_2,n)=O(n^{-2})$. 
\end{enumerate}

In the other cases some nice cancellation occurs between Wg functions. 
\begin{enumerate}[\rm(i)]
\setcounter{enumi}{2}
\item $\pi(\sigma_1,\sigma_2,\sigma_3)$ has two blocks with cardinality 1 and 3, say $\{\{1,3,4\},\{2\}\}$. The equivalence classes are 
$I_1 \cup I_3 \cup I_4$ and $I_2$. 
The only indices $A$'s for which $f_A$ is non-zero are $A=\{1,3,4\}, \{1,2,3,4\}$. By inspection we see that  
\begin{align*}
&f_{\{1,3,4\}}=-\Wg(\sigma_{3| I_2 },n) \Wg (\sigma_{3| I_1 \cup I_3 \cup I_4 },n),  \\
& f_{\{1,2,3,4\}}=\Wg(\sigma_3,n). 
\end{align*}
By \eqref{eq:decay} and the multiplicativity \eqref{Moeb} of Moebius functions, we obtain 
\[
\begin{split}
f&= f_{\{1,3,4\}}+f_{\{1,2,3,4\}} \\
&= n^{- 4 k -|\sigma_3|}(- \Moeb(\sigma_{3| I_2 }) \Moeb(\sigma_{3| I_1 \cup I_3 \cup I_4 })+\Moeb(\sigma_3)+O(n^{-2})) \\
&= O(n^{-4k-|\sigma_3|-2}). 
\end{split}
\]
Thus $\Tr_{\sigma_2}(B_1,\dots, B_{4k})f (\sigma_1,\sigma_2,n)=O(n^{-|\sigma_2|-|\sigma_3|-2}).$ 

\item $\pi(\sigma_1,\sigma_2,\sigma_3)$ has 3 blocks, say $\{\{1\},\{2\}, \{3,4\}\}$. 
The indices $A$'s for which $f_A$ is non-zero are $A=\{3,4\}, \{1,3,4\}, \{2,3,4\}, \{1,2,3,4\}$. We see that  
\[
\begin{split}
f
=& f_{\{3,4\}}+ f_{\{1,3,4\}}+ f_{\{2,3,4\}}+f_{\{1,2,3,4\}} \\
=& \Wg(\sigma_{3| I_1 },n) \Wg(\sigma_{3| I_2 },n) \Wg (\sigma_{3| I_3 \cup I_4 },n)-\Wg(\sigma_{3| I_2 },n) \Wg (\sigma_{3| I_1 \cup I_3 \cup I_4 },n)\\
&-\Wg(\sigma_{3| I_1 },n) \Wg (\sigma_{3| I_2 \cup I_3 \cup I_4 },n) + \Wg(\sigma_3,n) \\
=& n^{- 4 k -|\sigma_3|}\big( \Moeb(\sigma_{3| I_1 }) \Moeb(\sigma_{3| I_2 }) \Moeb (\sigma_{3| I_3 \cup I_4 })-\Moeb(\sigma_{3| I_2 }) \Moeb (\sigma_{3| I_1 \cup I_3 \cup I_4 })\\
&-\Moeb(\sigma_{3| I_1 }) \Moeb (\sigma_{3| I_2 \cup I_3 \cup I_4 }) + \Moeb(\sigma_3) +O(n^{-2})\big) \\
=& O(n^{-4k-|\sigma_3|-2}). 
\end{split}
\]

\item $\pi(\sigma_1,\sigma_2,\sigma_3)=\{\{1\},\{2\},\{3\},\{4\}\}$, namely, every $I_i$ is invariant under the actions of 
$\sigma_1,\sigma_2,\sigma_3$. In this case $f_A$ contribute to $f$  for all the 16 subsets $A \subset \{1,2,3,4\}$. 
By multiplicativity, the dominant contribution to $f$ is the sum of 16 M\"obius functions 
$$
\pm \Moeb(\sigma_{3| I_1 })\Moeb(\sigma_{3| I_2 })\Moeb(\sigma_{3| I_3 })\Moeb(\sigma_{3| I_4 })
$$ 
multiplied by $n^{- 4 k -|\sigma_3|}$. Exactly half of them have minus signs, so they cancel. Thus $f = O(n^{-4k-|\sigma_3|-2})$. 
\end{enumerate}
This concludes the proof of \eqref{eq:decay2}.
\end{proof}

\begin{theorem}\label{ACM3}
Let $A_i=A_i(n), B_j=B_j(n), i=1,\dots, k,$ $j=1,\dots, \ell$ be $n\times n$ deterministic matrices and let $U= U(n)$ be a 
Haar unitary random matrix. Suppose that 
\begin{enumerate}[\rm(1)] 
\item\label{A13} $((A_1,\dots, A_k), \Tr_n)$ converges in distribution to a $k$-tuple of trace class operators as $n\to\infty$,  
\item\label{B13} $((B_1,\dots, B_\ell), \tr_n)$ converges in distribution to an $\ell$-tuple of elements in a non-commutative 
probability space as $n\to\infty$. 
\end{enumerate}
Then the pair $(\{A_i\}_{i=1}^k, \{U B_j U^*\}_{j=1}^\ell)$ is asymptotically cyclically monotone almost surely with respect to $(\Tr_n,\tr_n)$. 
\end{theorem}
\begin{proof}
For notational convenience we write $U B_i U^*$ simply as $B_i$. 
From the arguments in the proofs of Theorem \ref{ACM},  and it suffices to show that when $k=\ell$ 
\be
\lim_{n\to\infty}\Tr_n(A_1 B_1 \cdots A_k B_k ) = \lim_{n\to\infty} \Tr_n(A_1\cdots A_k) \prod_{i=1}^k \lim_{n\to\infty} \tr_n(B_i) \quad \text{a.s.}
\ee
One may use the standard Borel-Cantelli argument, but a simpler argument is possible.  
By Lemma \ref{ACM2} and by monotone convergence we have 
\be
\bE\left[\sum_{n=1}^\infty\left|\Tr_n(A_1B_1\cdots A_kB_k)-\bE[\Tr_n(A_1B_1\cdots A_kB_k)]\right|^4\right]<\infty,
\ee
and so 
\be
\sum_{n=1}^\infty\left|\Tr_n(A_1B_1\cdots A_kB_k)-\bE[\Tr_n(A_1B_1\cdots A_kB_k)]\right|^4 <\infty\quad \text{a.s.,}
\ee
which implies that $\lim_{n\to\infty}\left|\Tr_n(A_1B_1\cdots A_kB_k) - \bE[\Tr_n(A_1B_1\cdots A_kB_k)]\right|=0$ a.s.  By Theorem \ref{ACM}  we know that 
\be
\lim_{n\to\infty}\bE[\Tr_n(A_1B_1\cdots A_kB_k)] =  \lim_{n\to\infty} \Tr_n(A_1\cdots A_k) \prod_{i=1}^k \lim_{n\to\infty} \tr_n(B_i), 
\ee
so we get the conclusion. 
\end{proof}

Proposition \ref{prop2.7} implies the following. Note that we can only conclude that the limiting operator is Hilbert Schmidt. 

\begin{corollary}\label{convergence EV} Under the assumptions of Theorem \ref{ACM3}, for any selfadjoint $\ast$-polynomial 
$P(x_1,\dots, x_k, y_1,\dots, y_\ell)$ such that $P(0,\dots, 0,y_1,\dots, y_\ell)=0$, the Hermitian random matrix 
$P(A_1, \dots, A_k, U B_1 U^*,\ldots, U B_\ell U^*)$ converges in eigenvalues to a selfadjoint Hilbert Schmidt operator almost surely. 
The limiting eigenvalues can be computed by using cyclic monotone independence. 
\end{corollary}
Examples of the limiting eigenvalues are computed in Section \ref{sec5}. 

Our result implies the almost sure version of Shlyakhtenko's asymptotic infinitesimal freeness. 

\begin{corollary}\label{IF2} Let $A_i=A_i(n), B_j=B_j(n), i=1,\dots, k,$ $j=1,\dots, \ell$ be $n\times n$ deterministic matrices and let $U= U(n)$ be a 
Haar unitary random matrix. In addition to the assumption \eqref{A13} of Theorem \ref{ACM3}, we assume that 
$((B_1,\dots, B_\ell),\tr_n)$ converges up to the first order to an $\ell$-tuple of elements in an infinitesimal non-commutative 
probability space (see Definition \ref{def convergenceIF}).
Then for any $\ast$-polynomial $P$ in the unital noncommutative $\ast$-polynomial ring $\cC:=\C[x_1,\dots, x_k, y_1,\dots, y_\ell]$ 
the limits 
\begin{align}
&\tau( P):=\lim_{n\to\infty}\tr_n(P(A_1,\dots, A_k, UB_1U^*,\dots, UB_\ell U^*)), \\
&\tau'( P):= \lim_{n\to\infty}\frac{\tr_n(P(A_1,\dots, A_k, U B_1U^*,\dots, U B_\ell U^*)) - \tau( P)}{1/n}
\end{align} 
exist almost surely, and thus $(\cC, \tau',\tau)$ is an infinitesimal non-commutative probability space. 
Moreover, $\{x_i\}_{i=1}^k, \{y_j\}_{j=1}^\ell$ are infinitesimally free with respect to $(\tau',\tau)$. 
\end{corollary}
\begin{proof}
We decompose  $P=Q+R$ where $Q=Q(\{x_i\}_{i=1}^k,\{y_j\}_{j=1}^\ell)$  and $R=R(\{y_j\}_{j=1}^\ell)$ such that 
$Q(0,\dots, 0, y_1,\dots, y_\ell)=0$. By Theorem \ref{ACM3} $\Tr_n(Q)$ converges almost surely to a finite real number, 
so $\tau(Q)=0$. Hence $\tau'(Q) = \lim_{n\to\infty} \Tr_n(Q)$ converges almost surely. Since $((B_1,\dots, B_\ell), \tr_n)$ 
converges up to the first order then $\tau( R), \tau'(R )$ converge too. Therefore the limits $\tau(P ), \tau'( P)$ exist. 

Let $\cA:=\C[x_1,\dots, x_k]_0$ (not containing the unit) and let $\cB:= \C[y_1,\dots, y_\ell]$ (containing the unit). 
Then $\ideal_\cB(\cA) \subset \ker(\tau)$ since, as we saw, $\tau(Q)=0$. Since $(\cA,\cB)$ is cyclically monotone by 
Theorem \ref{ACM3}, $\cA,\cB$ are infinitesimally free by Proposition \ref{PropIF}. 
\end{proof}
Through the calculation in \cite{Shl} Shlyakhtenko suggested that infinitesimal freeness is applicable to outliers, but a 
rigorous proof is not obtained yet. 
\begin{problem} Combining the calculation of Shlyakhtenko \cite{Shl} and our almost sure convergence (and other ideas if needed), 
is it possible to rigorously prove the phase transition phenomena of outliers found by 
Baik et al.\ \cite{BBAP2005}, P\'ech\'e \cite{Pec2006}, and more generally by Benaych-Georges and 
Nadakuditi \cite{BGN2011} and Belinschi et al.\ \cite{BBCF}? 
\end{problem}
\subsection{General compact case}

Actually, we do not need trace class distributions in order to obtain the  almost sure convergence of eigenvalues, 
the compact setup is enough. In this section we denote by $\|\cdot\|$ the operator norm on $M_n(\C)$. 

\begin{theorem} \label{enhancedACM}
Let $A_i =A_i(n), B_j=B_j(n), i=1,\dots,k$, $j=1,\dots, \ell$ be deterministic $n\times n$ matrices and $U=U(n)$ be an $n\times n$ Haar unitary such that 
\begin{enumerate}[\rm(1)] 
\item $A_1,\dots, A_k$ are Hermitian, 
\item\label{bounded1} $((A_1,\dots, A_k), \Tr_n)$ converges in compact distribution to a $k$-tuple of compact 
operators $((a_1,\dots, a_k), \Tr_H)$ as $n\to\infty$ (see Definition \ref{def:cv-compact}),  
\item $((B_1,\dots, B_\ell), \tr_n)$ converges in distribution to an $\ell$-tuple of elements in a non-commutative probability space as $n\to\infty$,   
\item\label{bounded} $\sup_{n\in\N}\|B_i(n)\|<\infty$ for every $i=1,\dots,\ell$.   
\end{enumerate}
Let $P(x_1,\dots, x_k,y_1,\dots, y_\ell)$ be a selfadjoint $\ast$-polynomial with selfadjoint variables $x_1,\dots, x_k$ such 
that $P(0,\dots, 0, y_1,\dots, y_\ell)=0$. 
Then 
$
P(A_1,\dots, A_k, U B_1 U^*,\dots, UB_\ell U^*)
$
converges in eigenvalues to a deterministic compact operator almost surely. 
\end{theorem} 
\begin{proof} We may assume that $k=\ell$. For simplicity $U B_p U^*$ is abbreviated to $B_p$. By assumption, 
every $A_p$ converges in eigenvalues. We can then find some sequence $\{\ep_j\}_{j \geq1}$ such that $\ep_j \downarrow0$ and 
$$
\{\ep_j: j\in\N \}\cap \{|\lambda_i^u(A_p(n))|, \lim_{N\to\infty}|\lambda_i^u(A_p(N))|: i,n\in\N, 1 \leq p \leq k, u\in\{+,-\} \}=\emptyset.
$$
 Let $f_j$ be a continuous function on $\R$ such that $f_j$ is non-decreasing and 
\be
f_j(x)= 
\begin{cases} 0, & |x|< \ep_{j+1}, \\
x,& |x| > \ep_j. 
\end{cases}
\ee
Let $A_p^{(j)}, a_p^{(j)}$ be the truncations $f_j(A_p), f_j(a_p)$ respectively, so that they are finite rank operators and in particular 
trace class operators. By the definition of convergence in compact distribution, $((A_1^{(j)}, \dots, A_k^{(j)}), \Tr_n)$ converges 
in distribution to the trace class operators $((a_1^{(j)}, \dots, a_k^{(j)}), \Tr_H)$ as $n\to\infty$.  Under such circumstances, for 
each fixed $j \in\N$ we apply Corollary \ref{convergence EV} to 
the pair
$
(\{A_1^{(j)},\dots A_k^{(j)}\}, \{B_1, \dots, B_k\}). 
$
Thus,  the random eigenvalues of the polynomial $P(A_1^{(j)},\dots, A_k^{(j)}, B_1,\dots, B_k)$,  denoted by 
$\{\lambda_i^{(j)}(n)\}_{i\geq1}$, converge to some deterministic eigenvalues $\{\lambda_i^{(j)}\}_{i\geq1} \in \ell^2(\R)$ as $n\to\infty$: 
\be
\lim_{n\to\infty}(\lambda_i^{(j)})^\pm(n) = (\lambda_i^{(j)})^\pm,\qquad i,j \in\N. 
\ee

It is follows by functional calculus that 
\begin{align}
&\sup_{n\in\N, 1\leq p \leq k}\left\|A_p^{(j)} - A_p\right\| \leq \ep_j, && j \in \N, \label{eq44} \\
&\sup_{n\in\N, 1\leq p \leq k}\left\|A_p^{(j)} - A_p^{(j')}\right\| \leq \ep_j, && 1\leq  j \leq j'.  \label{eq45}
\end{align}
After the use of several triangular inequalities, we can show by \eqref{eq44}, \eqref{eq45} and the assumption 
\eqref{bounded} that the random variables
\begin{align}
&\delta_j:= \sup_{n\in\N}\left\|P(A_1,\dots, A_k, B_1,\dots, B_k)-P(A_1^{(j)},\dots, A_k^{(j)}, B_1,\dots, B_k)\right\|, && 1 \leq j, \\
&\delta_{j, j'}:= \sup_{n\in\N}\left\|P(A_1^{(j)},\dots, A_k^{(j)}, B_1,\dots, B_k)-P(A_1^{(j')},\dots, A_k^{(j')}, B_1,\dots, B_k)\right\|,&& j \leq j'
\end{align}
converge to $0$ almost surely as $j\to\infty$. 

Let $\{\lambda_i(n)\}_{i\geq1}$ be the random eigenvalues of $P(A_1,\dots, A_k, B_1,\dots, B_k)$. By Weyl's inequality for 
eigenvalues \cite[Corollary III.2.6.]{Bhatia}, we have 
\begin{align}
&\left| \lambda_i^\pm(n)- (\lambda_i^{(j)})^\pm(n)\right| \leq \delta_j \text{~a.s.},&& n \in\N, \label{first} \\
& \left| (\lambda_i^{(j)})^\pm(n)- (\lambda_i^{(j')})^\pm(n)\right| \leq \delta_{j,j'} \text{~a.s.}, &&n \in\N. \label{second}
\end{align}
The second inequality \eqref{second} gives us  $|(\lambda_i^{(j)})^\pm- (\lambda_i^{(j')})^\pm| \leq \delta_{j,j'}$ for 
$j \leq j'$, and so $\{(\lambda_i^{(j)})^\pm\}_{j\geq1}$ is a Cauchy sequence and has a limit $\lambda_i^\pm$ as $j\to\infty$ for 
each fixed $i$. 
The first inequality \eqref{first} gives us 
\be
\begin{split}
\varlimsup_{n\to\infty}\left|\lambda_i^\pm(n)-\lambda_i^\pm\right| 
\leq& \varlimsup_{n\to\infty}\left|\lambda_i^\pm(n)-(\lambda_i^{(j)})^\pm(n)\right| +\varlimsup_{n\to\infty}\left|(\lambda_i^{(j)})^\pm(n)-(\lambda_i^{(j)})^\pm\right| \\
& +\left|(\lambda_i^{(j)})^\pm-\lambda_i^\pm\right| \\ 
\leq& \delta_j+\left|(\lambda_i^{(j)})^\pm-\lambda_i^\pm\right| \text{~a.s.}  
\end{split}
\ee
By letting $j\to\infty$ we get the almost sure convergence $\lambda_i^\pm(n)\to \lambda_i^\pm$ as $n\to\infty$. 

Finally we prove that $\lambda_i^\pm \to 0$ as $i\to\infty$, so the limiting eigenvalues correspond to a compact operator. 
We denote by  $s_1(X)\geq s_2(X)\geq \dots \geq0$ the singular values of a compact operator $X$. Note that in our notation 
of proper arrangement, for a selfadjoint operator $X$ it holds that $s_i(X)=|\lambda_i(X)|$.  Note also that singular values 
satisfy $s_i(X Y Z) \leq \|X\| \|Z\| s_i(Y)$ and $s_{i+j-1}(X+Y) \leq s_i(X) + s_i(Y)$ which can be proved by the mini-max 
principle \cite[Corollary III.1.2 and Problem III.6.2]{Bhatia}.

Suppose that the polynomial $P$ is of the form
\be
P(A_1,\dots, A_k,B_1,\dots, B_k)=\sum_{\ell=1}^m X_\ell, 
\ee
where $X_\ell$ is a monomial containing some $A_{p(\ell)}$.
Take $\ep>0$. By assumption \eqref{bounded1}, there exists some $i_0\in\N$ such that 
\be
\sup_{n\in\N, 1 \leq p \leq k}|\lambda_i(A_p)| \leq \ep/m,\qquad i \geq i_0. 
\ee 
Let $X_\ell'$ be the monomial $X_\ell$ with the matrix $A_{p(\ell)}$ removed from it. Then we get 
\begin{align}
&s_{m i - m+1}( P) \leq \sum_{\ell=1}^m \|X_\ell' \| s_i(A_{p(\ell)}) \leq M \ep,  \qquad i \geq i_0, \\
&M :=\sup_{n\in\N, 1 \leq \ell \leq k} \|X_\ell' \|.   
\end{align}
Since $\sup_{n\in\N}\|A_p(n)\|<\infty$ and  $\sup_{n\in\N}\|B_p(n)\|<\infty$ then $M$ is finite almost surely too. 
This shows that $\sup_{n\in\N} s_i( P) = \sup_{n\in\N}|\lambda_i(n)|$ converges to $0$ as $i\to\infty$ almost surely, 
and so $\lim_{i\to\infty}\lambda_i=0$.  
\end{proof}

\subsection{Several Haar unitaries case}\label{Several}
Theorems \ref{ACM}, \ref{ACM3}, \ref{enhancedACM} and Corollaries \ref{convergence EV}, \ref{IF2} can be generalized 
to the case when several independent Haar unitaries are involved. Proofs are just to combine our results of asymptotic cyclic 
monotonicity with asymptotic freeness between $B_i$'s. For example Theorem \ref{ACM}  can be generalized as follows. 

\begin{theorem} \label{GACM}
Let $A_i=A_i(n), B_{i j}=B_{i j}(n), i,j=1,\dots,k$ be $n\times n$ random matrices and $U_i= U_i(n), i=1,\dots, k$ be 
independent Haar unitary random matrices. Suppose that 
\begin{enumerate}[\rm(1)]
\item\label{A14} $((A_1,\dots, A_k), \bE\otimes \Tr_n)$ converges in distribution to a $k$-tuple of trace class operators as $n\to\infty$,  
\item\label{B14} for each $i$, $((B_{i1}, \dots, B_{i k}), \bE\otimes \tr_n)$ converges in distribution to a $k$-tuple of 
elements in a non-commutative probability space as $n\to\infty$,  
\item the families $\{A_i\}_{i=1}^k, \{B_{ij}\}_{i,j=1}^k, \{U_i\}_{i=1}^k$ are independent. 
\end{enumerate}
Then the pair $(\{A_i\}_{i=1}^k, \{U_{i} B_{i j} U_{i}^*\}_{i,j=1}^k)$ is asymptotically cyclically monotone with respect 
to $(\bE \otimes \Tr_n,\bE\otimes \tr_n)$.  
\end{theorem} 
\begin{proof}
We take a Haar unitary $U$ independent of all $A_i, B_{ij}$ and $U_i$. Let $\tilde{B}_{ij}:= U_{i} B_{ij} U_{i}^*$. 
By \cite[Theorem 3.1]{Collins2003} it follows that $\{\tilde{B}_{1j}\}_{j=1}^k, \dots, \{\tilde{B}_{kj}\}_{j=1}^k$ are asymptotically 
free with respect to $\bE\otimes \tr_n$, so $((\tilde{B}_{11}, \tilde{B}_{12}, \dots, \tilde{B}_{kk}), \bE\otimes \tr_n)$ converges 
in distribution to a $k^2$-tuple of elements in a non-commutative probability space as $n\to\infty$. Theorem \ref{ACM} implies 
that the pair $(\{A_i\}_{i=1}^k, \{U \tilde{B}_{ij} U^*\}_{i,j=1}^k)$ is asymptotically cyclically monotone with respect 
to $(\bE \otimes \Tr_n,\bE\otimes \tr_n)$. Since $(U U_1, \dots, U U_k)$ has the same distribution as $(U_1, \dots, U_k)$, 
we conclude that the pair $(\{A_i\}_{i=1}^k, \{\tilde{B}_{ij}\}_{i,j=1}^k)$ is also asymptotically cyclically monotone with 
respect to $(\bE \otimes \Tr_n,\bE\otimes \tr_n)$. 
\end{proof}

The same technique allows us to generalize Lemma \ref{ACM2} to the several Haar unitaries case, and so we obtain the 
almost sure convergence, namely the generalization of Theorem \ref{ACM3}. 
\begin{theorem} \label{GACM3}
Let $A_i=A_i(n), B_{i j}=B_{i j}(n), i,j=1,\dots,k$ be $n\times n$ deterministic matrices and $U_i= U_i(n), i=1,\dots, k$ be 
independent Haar unitary random matrices. Suppose that 
\begin{enumerate}[\rm(1)]
\item\label{A0} $((A_1,\dots, A_k), \Tr_n)$ converges in distribution to a $k$-tuple of trace class operators as $n\to\infty$,  
\item\label{B0} for each $i$, $((B_{i1}, \dots, B_{i k}), \tr_n)$ converges in distribution to a $k$-tuple of elements in a 
non-commutative probability space as $n\to\infty$,  
\end{enumerate}
Then the pair $(\{A_i\}_{i=1}^k, \{U_{i} B_{i j} U_{i}^*\}_{i,j=1}^k)$ is asymptotically cyclically monotone with respect 
to $(\Tr_n,\tr_n)$ almost surely.  
\end{theorem}

Corollary \ref{convergence EV} is generalized in the following form. 

\begin{corollary}\label{G convergence EV} Under the assumptions of Theorem \ref{GACM3}, for any 
selfadjoint $\ast$-polynomial $P(\{x_i\}_{i=1}^k,\{y_{ij}\}_{i,j=1}^k)$ such that $P(\{0\}_{i=1}^k,\{y_{ij}\}_{i,j=1}^k)=0$, the 
Hermitian random matrix $P(\{A_i\}_{i=1}^k, \{U_i B_{ij} U_i^*\}_{i,j=1}^k)$ converges in eigenvalues to a selfadjoint 
Hilbert Schmidt operator almost surely. The limiting eigenvalues can be computed by asymptotic cyclic monotonicity 
of $(\{A_i\}_{i=1}^k, \{U_i B_{ij} U_i^*\}_{i,j=1}^k)$ and asymptotic freeness of 
$\{ \{U_1 B_{1j} U_1^*\}_{j=1}^k, \dots, \{U_k B_{kj} U_k^*\}_{j=1}^k \}$. 
\end{corollary}

Corollary \ref{IF2} and Theorem \ref{enhancedACM} can be similarly generalized, the explicit statements of which are omitted.  
We also mention that we can obtain the above results if we take $B_{ij}=G_i$ where $G_1,\dots, G_k$ are 
independent GUEs normalized so that each $G_i$ converges in distribution. 
In this case we may remove the Haar unitaries $U_i$ since the independent GUEs provide independent 
Haar unitaries when diagonalized.

In the case of a single Haar unitary, considering the pair $(\{U A_i U^*\}_{i=1}^k, \{B_j\}_{j=1}^\ell)$ is equivalent to 
considering the pair $(\{A_i\}_{i=1}^k, \{U B_j U^*\}_{j=1}^\ell)$. 
However, in the several Haar unitaries case, the pair $(\{U_i A_{i j} U_i^*\}_{i,j=1}^k,\{B_i\}_{i=1}^k)$ becomes rather trivial. 

\begin{proposition}\label{trivial} Let $A_i=A_i(n), B_i=B_i(n), i=1,\dots,k$ be $n\times n$ deterministic matrices 
and $U_i= U_i(n), i=1,\dots, k$ be independent Haar unitaries. Suppose that 
\begin{enumerate}[\rm(1)]
\item\label{A1} $((A_1,\dots, A_k), \Tr_n)$ converges in distribution to a $k$-tuple of trace class operators as $n\to\infty$,  
\item\label{B1} $((B_{1}, \dots, B_{k}), \tr_n)$ converges in distribution to a $k$-tuple of elements in a non-commutative 
probability space as $n\to\infty$,  
\item $\sup_{n\in\N} \|B_i(n)\|<\infty$ for every $i =1,\dots,k.$
\end{enumerate}
For any tuple $(i_1,\dots,i_k)\in\{1,\dots,k\}^k$ such that at least two of them are distinct, we have the almost sure convergence 
\[
\lim_{n\to\infty}\Tr_n(U_{i_1}A_{1}U_{i_1}^* B_1 \cdots U_{i_k} A_k U_{i_k}^* B_k)=0. 
\]
\end{proposition} 
\begin{proof} We may suppose that $1,2 \in \{i_1,\dots, i_k\}$. By $\bE^{U_1}$ we denote the expectation with respect 
to $U_1$, leaving $U_2, \dots, U_k$ unchanged (that is, the conditional expectation onto the $\sigma$-field generated 
by $U_2,\dots, U_k$). We show that  almost surely  
\be\label{eq54}
\bE^{U_1}\!\left[\left|\Tr_n(U_{i_1}A_{1}U_{i_1}^* B_1 \cdots U_{i_k} A_k U_{i_k}^* B_k)\right|^2\right] =O(n^{-2}). 
\ee
 By the cyclic property of the trace and by the obvious property $\overline{\Tr_n(X)}=\Tr_n(X^*)$, the LHS of \eqref{eq54} equals    
\be\label{eq55}
\begin{split}
\bE^{U_1}[\Tr_n(C_1 U_1 D_1U_1^* \cdots C_\ell U_1 D_\ell U_1^* ) \Tr_n(C_{\ell+1} U_1 D_{\ell+1}U_1^* \cdots C_{2\ell} U_1 D_{2\ell} U_1^* )], 
\end{split}
\ee
where $C_i$'s are products of $B_p, B_p^*, U_{q} A_r U_{q}^*$ and $U_{q} A_r^* U_{q}^*$ with $q\geq2$ 
and $D_i \in\{A_r^{\ep}: \ep\in\{1,*\}, 1 \leq r \leq k\}$. At least one of the matrices $C_1,\dots, C_\ell$ must have a 
factor $U_2 A_r U_2^*$ for some $r$ and similarly for $\{C_{\ell+1},\dots, C_{2\ell} \}$. Let $Z^{\cup2}$ be the 
permutation $(1,\dots, \ell)(\ell+1,\dots, 2\ell)$. By the Weingarten formula \eqref{Wg2}, the quantity \eqref{eq55} equals 
\be
\sum_{\substack{\sigma_1,\sigma_2,\sigma_3\in S_{2\ell}\\ \sigma_1\sigma_2\sigma_3=Z^{\cup2}}}\Tr_{\sigma_1}(C_1, \dots, C_{2\ell})\Tr_{\sigma_2}(D_1,\dots, D_{2\ell}) \Wg(\sigma_3,n). 
\ee
By assumption $\Tr_{\sigma_2}(D_1,\dots, D_{2\ell})=O(1)$ and by inspection 
$\Tr_{\sigma_1}(C_1,\dots, C_{2\ell}) = O(n^{2\ell -2})$. By \eqref{eq:decay} $\Wg(\sigma_3,n)$ 
behaves as $O(n^{-2\ell-|\sigma_3|})$. Thus, we obtain the behavior $O(n^{-2})$ of \eqref{eq54}. By taking the 
sum $\sum_{n=1}^\infty$ and by using the conditional monotone convergence theorem we obtain the conclusion. 
 \end{proof}
Therefore, a monomial of $\{U_i A_i U_i^*\}_{i=1}^k,\{B_i\}_{i=1}^k$ is trivial with respect to $\Tr_n$ if it contains at least two 
distinct factors from $\{U_i A_{i} U_i^*\}_{i=1}^k$. The nontrivial case is only when one factor from $\{U_i A_{i} U_i^*\}_{i=1}^k$ 
appears in each monomial; for instance  
\be
\sum_{i=1}^k B_i U_i A_i U_i^* B_i^*. 
\ee
We compute the eigenvalues of this model in Proposition \ref{enhancedACM3}.

\section{Some concrete computation of eigenvalues}\label{sec5}

\subsection{Eigenvalues of polynomials of random matrices}

We provide explicit discrete eigenvalues of some polynomials of random matrices converging to compact operators. 
Results in Section \ref{sec:rmt} show that the computation of the eigenvalues reduces to the eigenvalues of polynomials 
of cyclically monotone elements. Then the computations in Section \ref{secEVCM} give the corresponding results on large 
random matrices converging to trace class operators. We can then show the results in the compact setup (i.e.\ the same 
assumptions as in Theorem \ref{enhancedACM}) by approximation.

\begin{theorem} \label{enhancedACM2} 
Let $A_i =A_i(n), B_i=B_i(n), i=1,\dots,k$ be deterministic $n\times n$ matrices and $U=U(n)$ be an $n\times n$ Haar unitary such that 
\begin{enumerate}[\rm(1)] 
\item $A_1,\dots, A_k$ are Hermitian, 
\item\label{bounded1-2} $((A_1,\dots, A_k), \Tr_n)$ converges in compact distribution to a $k$-tuple of compact 
operators $((a_1,\dots, a_k), \Tr_H)$ as $n\to\infty$ (see Definition \ref{def:cv-compact}),  
\item $((B_1,\dots, B_k), \tr_n)$ converges in distribution to a $k$-tuple of elements in a non-commutative probability space as $n\to\infty$,   
\item\label{bounded-2} $\sup_{n\in\N}\|B_i(n)\|<\infty$ for every $i=1,\dots,k$.   
\end{enumerate}
In this case let $\beta_i:= \lim_{n\to\infty} \tr_n(B_i)$, $\beta_{i j}:= \lim_{n\to\infty}\tr_n(B_i^* B_j)$ and 
$B:= (\beta_{i j})_{i,j=1}^k$. The following statements hold true.

\begin{enumerate}[\rm(i)] 
\item\label{A2} We have 
\[
\lim_{n\to\infty}\EV\!\left(\sum_{i=1}^k UB_i U^* A_i (U B_i U^*)^\ast\right) = \EV\!\left(\sqrt{B} \diag(a_1,\dots, a_k) \sqrt{B}\right) \text{~a.s.},
\]
 where $\sqrt{B} \diag(a_1,\dots, a_k) \sqrt{B}$ is viewed as an element of $(M_k(\C) \otimes\cS^1(H), \Tr_k \otimes \Tr_H)$. 

\item\label{D2}  Suppose that $B_1,\dots, B_k$ are Hermitian. Then  
\[
\lim_{n\to\infty}\EV\!\left(\sum_{i=1}^k A_i U B_i U^* A_i\right) = \EV\!\left(\sum_{i=1}^k \beta_i a_i^2\right) \text{~a.s.}
\]

\item\label{B2} Suppose that $k=1$ and $B_1$ is Hermitian. Let $p = \sqrt{\beta_{11}} + \beta_1, q = -(\sqrt{\beta_{11}}-\beta_1).$ 
 Then 
 $$
 \lim_{n\to\infty} \EV(A_1 UB_1U^* + UB_1U^* A_1) = (p \EV(a_1) ) \sqcup (q \EV(a_1) ) \text{~a.s.}
 $$ 

\item\label{C2} Suppose that $k=1$ and $B_1$ is Hermitian. Let $r:= \sqrt{\beta_{11} -\beta_1^2}$. Then 
$$
\lim_{n\to\infty}\EV(\ri [A_1, UB_1U^*]) = (r \EV(a_1 ) ) \sqcup (-r \EV(a_1) ) \text{~a.s.}
$$
\end{enumerate}
\end{theorem} 
\begin{proof} 
We only show \eqref{A2} since the other cases are similar. 
We reuse the notations and proof of Theorem \ref{enhancedACM}. In particular $U B_i U^*$ is abbreviated to $B_i$. Now the polynomial $P$ is 
\be
P(x_1,\dots, x_k, y_1,\dots, y_k) = \sum_{i=1}^k y_i x_i y_i^*. 
\ee
Recall that 
\begin{align}
&\{\lambda_i(n)\}_{i\geq1}=\EV\left(P(A_1,\dots, A_k, B_1,\dots, B_k)\right), \\
&\{\lambda_i\}_{i\geq1}=\lim_{n\to\infty}\EV\left(P(A_1,\dots, A_k, B_1,\dots, B_k) \right) \text{~a.s.}, \\
&\{\lambda_i^{(j)}(n)\}_{i\geq1} = \EV\!\left(P(A_1^{(j)},\dots, A_k^{(j)}, B_1,\dots, B_k)\right), \\ 
&\{\lambda_i^{(j)}\}_{i\geq1}=\lim_{n\to\infty}\EV\!\left(P(A_1^{(j)},\dots, A_k^{(j)}, B_1,\dots, B_k) \right) \text{~a.s.}
\end{align}
From the result in the trace class setup in Theorem \ref{compute} and asymptotic cyclic monotonicity in 
Corollary \ref{convergence EV}, we have the identity
\be\label{eq567}
\{\lambda_i^{(j)}\}_{i\geq1} = \EV\!\left(\sqrt{B} \diag(a_1^{(j)},\dots, a_k^{(j)}) \sqrt{B}\right). 
\ee
Now we define another sequence of eigenvalues 
\be
\{\lambda_i'\}_{i\geq1}:= \EV\!\left(\sqrt{B} \diag(a_1,\dots, a_k) \sqrt{B}\right). 
\ee
Our goal is to demonstrate that $\lambda_i=\lambda_i'$ for every $i\in\N$.  For this we use the inequality  
\be
\begin{split}
\left|\lambda_i^\pm -(\lambda_i')^\pm\right| 
\leq & \left|\lambda_i^\pm -\lambda_i^\pm(n)\right| +\left|\lambda_i^\pm(n)-(\lambda_i^{(j)})^\pm(n) \right| \\
&+\left|(\lambda_i^{(j)})^\pm(n) - (\lambda_i^{(j)})^\pm\right| + \left| (\lambda_i^{(j)})^\pm- (\lambda_i')^\pm\right|.   
\end{split}
\ee
We proved in the proof of Theorem \ref{enhancedACM} that the first term on the RHS converges to $0$ as $n\to\infty$ and 
proved in \eqref{first} that the second term is bounded by $\delta_j$ uniformly on $n$. The third term converges to $0$ 
as $n\to\infty$ from the result in the trace class setup. Taking $\varlimsup_{j\to\infty}\varlimsup_{n\to\infty}$ we get 
\be
\left|\lambda_i^\pm -(\lambda_i')^\pm\right| 
\leq \varlimsup_{j\to\infty}\left| (\lambda_i^{(j)})^\pm- (\lambda_i')^\pm\right|.   
\ee
By \eqref{eq567} it now suffices to show that 
\be
\lim_{j\to\infty} \EV\!\left(\sqrt{B} \diag(a_1^{(j)},\dots, a_k^{(j)}) \sqrt{B}\right) = \EV\!\left(\sqrt{B} \diag(a_1,\dots, a_k) \sqrt{B}\right), 
\ee
which follows from the fact 
\be
\lim_{j\to\infty}\left\|\sqrt{B} \diag(a_1^{(j)},\dots, a_k^{(j)}) \sqrt{B} - \sqrt{B} \diag(a_1,\dots, a_k) \sqrt{B}\right\| = 0 
\ee
and Weyl's inequality \cite[Corollary III.2.6]{Bhatia} which bounds the difference of eigenvalues by the operator 
norm of the difference of elements. 
\end{proof}

When several independent Haar unitaries are involved we can still compute the eigenvalues. For example we obtain the following. 
\begin{proposition} \label{enhancedACM3} 
Under the assumptions of Theorem \ref{enhancedACM2}, we take independent Haar unitaries $U_i=U_i(n), i=1,\dots, k$.  
\begin{enumerate}[\rm(i)]
\item\label{A6} Let $\beta_i:= \lim_{n\to\infty} \tr_n(B_i)$. Then 
\[
\lim_{n\to\infty}\EV\!\left(\sum_{i=1}^k U_iB_i U_i^* A_i (U_i B_i U_i^*)^\ast\right) = \EV\!\left(\sum_{i=1}^k |\beta_i|^2 a_i\right) \text{~a.s.}
\]
\item\label{B6} Let $\gamma_i:= \lim_{n\to\infty} \tr_n(B_i^*B_i)$. Then we have 
\[
\lim_{n\to\infty}\EV\!\left(\sum_{i=1}^k B_i U_i A_i U_i^* B_i^*\right) = \EV(\gamma_1 a_1) \sqcup \cdots \sqcup \EV(\gamma_k a_k) \text{~a.s.}  
\]
\end{enumerate}
\end{proposition} 
\begin{proof} Suppose that $((A_1,\dots, A_k), \Tr_n)$ converges in distribution to trace class operators; 
the general compact case is proved by approximation. 

\eqref{A6} By Corollary \ref{G convergence EV}, the computation formula is obtained by the asymptotic cyclic 
monotonicity of $(\{A_i\}_{i=1}^k, \{U_i B_i U_i^* \}_{i=1}^k)$ and then the asymptotic freeness of 
$U_1 B_1 U_1^*$, $\dots$, $U_k B_k U_k^*$. This implies that we only need to replace the covariance matrix $B$ in 
Theorem \ref{enhancedACM2}\eqref{A2} with $B' = (\overline{\beta_i}\beta_j)_{i,j=1}^k \in M_k(\C)$.  The limiting eigenvalues 
are given by those of $\sqrt{B'} \diag(a_1,\dots, a_k) \sqrt{B'}$. Since $\sqrt{B'}= (\sum_{i=1}^k |\beta_i|^2)^{-1/2} B'$, 
the conclusion follows easily.

\eqref{B6} By Proposition \ref{trivial}, for each $\ell \in \N$ we have 
\be
\begin{split}
\Tr_n\!\left(\left(\sum_{i=1}^k B_i U_i A_i U_i^* B_i^*\right)^\ell\right) 
&= \sum_{i=1}^k \Tr_n\!\left(\left(B_i U_i A_i U_i^* B_i^*\right)^\ell\right)+o(1)\\
&= \sum_{i=1}^k \Tr_n\!\left(\left(A_i (U_i^* B_i^*B_i U_i)\right)^\ell\right)+o(1). 
\end{split}
\ee
By the cyclic monotonicity of the pair $(\{A_i\}_{i=1}^k, \{U_i^* B_i^*B_i U_i\}_{i=1}^k)$ (see Theorem \ref{ACM3}), 
we have the almost sure convergence 
\be
\begin{split}
\lim_{n\to\infty}\Tr_n\!\left(\left(\sum_{i=1}^k B_i U_i A_i U_i^* B_i^*\right)^\ell\right) 
&= \sum_{i=1}^k \lim_{n\to\infty}\Tr_n(A_i^\ell) \lim_{n\to\infty}\tr_n(B_i^*B_i)^\ell \\
&= \sum_{i=1}^k \Tr_H( (\gamma_i a_i)^\ell), 
\end{split}
\ee
and the conclusion follows by Proposition \ref{prop2.6}. 
\end{proof}

\subsection{Numerical illustration}
In this subsection, we give numerical illustrations of main theorems of this paper.
\begin{example}

Let $\{Z_{i}\}_{i}$ be a family of independent $n\times n$ non-selfadjoint Gaussian random matrices, 
that is, each $Z_i$ has entries that are independent identically distributed with the standard complex normal distribution. 
We consider a sample covariance matrix $X_{i}=Z_{i}Z_{i}^{\ast}/(2n)$ 
and a diagonal matrix $D=\diag(2^{-1},2^{-2},2^{-3},\dots,2^{-n})$.
From Theorem \ref{ACM}, the matrices $D$ and $X$ are asymptotically cyclically monotone. Table \ref{moments} 
shows numerical simulations of mixed moments and moments decomposed by cyclic monotone independence  when $n=500$.
\begin{table}[htbp]
\caption{Moments ($n=500$)}
\begin{center}
\begin{tabular}{|c|c||c|c|}
\hline
$\Tr(DX_1)$ & 1.65916 &$\Tr(D)\tr(X_1)$&1.64044\\
\hline
$\Tr(DX_1D)$ & 1.11507&$\Tr(D^{2})\tr(X_1)$ &1.08068 \\
\hline
$\Tr(DX_1DX_2)$ & 1.08634&$\Tr(D^{2})\tr(X_1)\tr(X_2)$&1.08034\\
\hline
$\Tr(DX_1DX_2D)$ & 1.02348 &$\Tr(D^{3})\tr(X_1)\tr(X_2)$&1.01548 \\
\hline
$\Tr(DX_1DX_2DX_3)$ & 1.00131   &$\Tr(D^{3})\tr(X_1)\tr(X_2)\tr(X_3)$&1.0176 \\
\hline
\end{tabular}
\end{center}
\label{moments}
\end{table}
\end{example}

\begin{example} 
We give numerical realization of Proposition \ref{enhancedACM3}\eqref{B6}.  
Let $D$ be the $n\times n$ diagonal matrix $\diag(2^{-1},2^{-2},\dots,2^{-n}) $ and 
$U_i,i=1,2,\dots$ be $n\times n$ independent Haar unitary random matrices.
Then the eigenvalues of $D+\sum_{i=1}^kU_iD U_i^\ast$ can be computed by Proposition \ref{enhancedACM3} and converge to 
\be
(\underbrace{2^{-1},\dots, 2^{-1}}_{\text{$k$ times}},\underbrace{2^{-2},\dots, 2^{-2}}_{\text{$k$ times}},\dots). 
\ee 
In Figure \ref{fig} we show the first 14 eigenvalues for $n=1000$.

\begin{figure}[htbp]
\includegraphics[width=5cm]{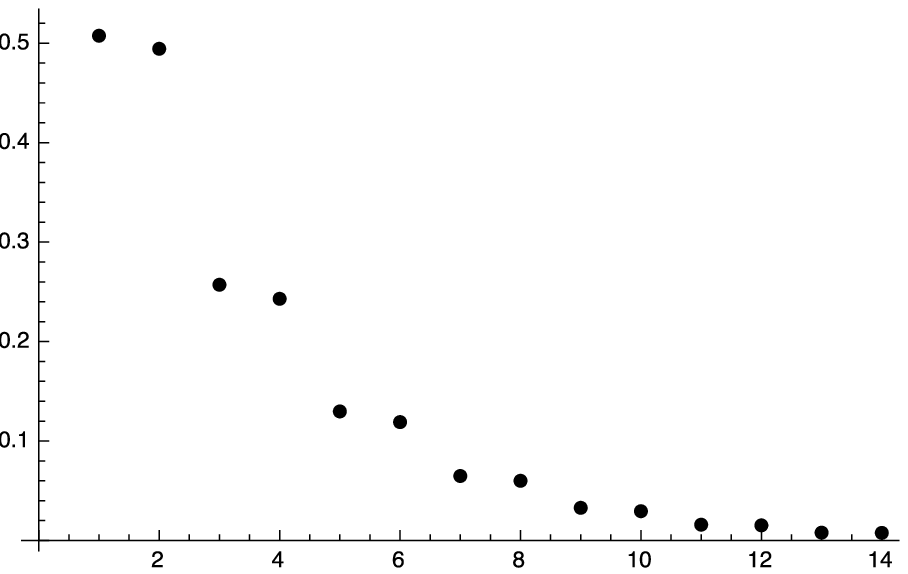}
\includegraphics[width=5cm]{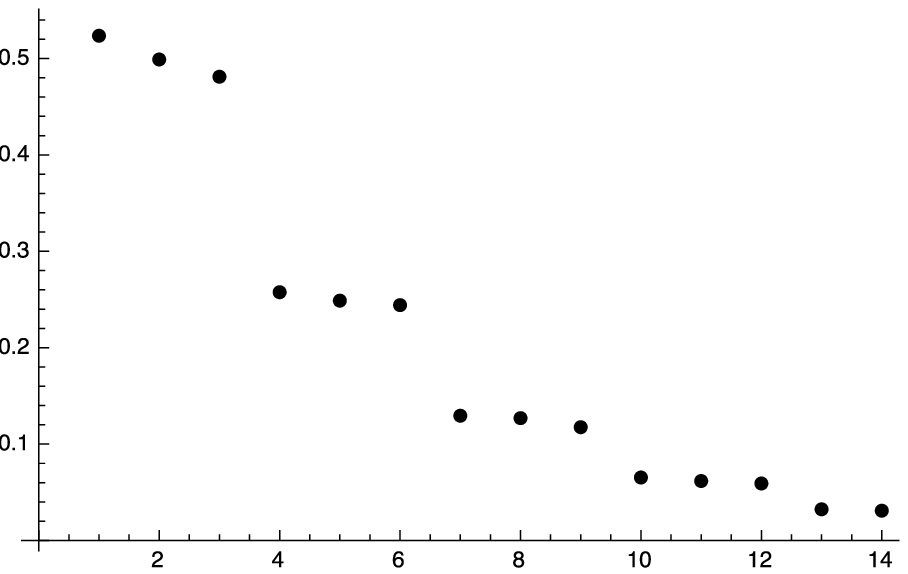}
\includegraphics[width=5cm]{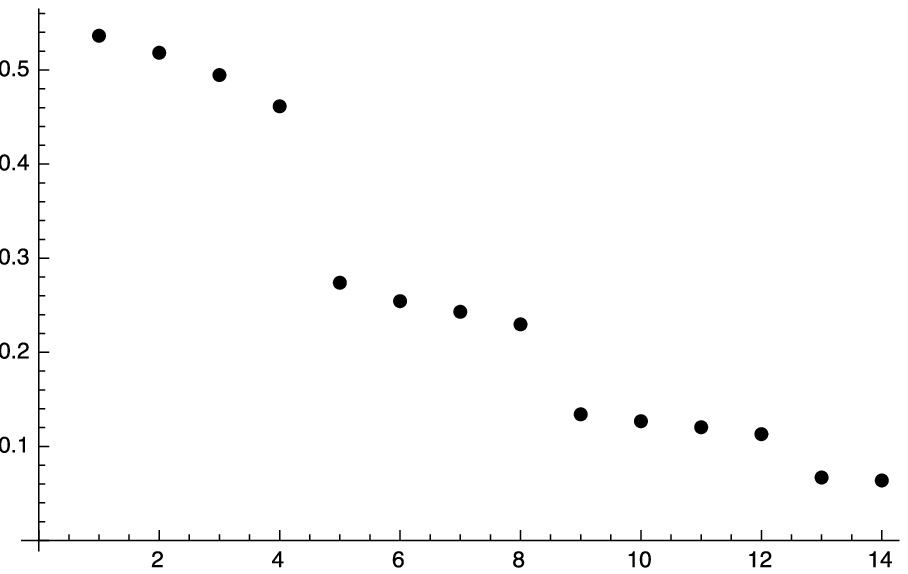}
\caption{Eigenvalues $\lambda_1,\dots, \lambda_{14}$ of $\left\{D+\sum_{i=1}^{k}U_i D U_i^* \right\}_{k=1,2,3}$ ($n=1000$)}
\label{fig}
\end{figure}

\end{example}

\begin{example}
We compute random matrix models of Theorem \ref{enhancedACM2} \eqref{B2} and \eqref{C2}.
Assume that $n=300$. Let $D$ be the $n\times n$ diagonal matrix $\diag(2^{-1},2^{-2},\dots,2^{-n}) $ 
and $U$ be an $n\times n$ Haar unitary. Let $Z$ be an $n\times n$ non-selfadjoint Gaussian random matrix.
We consider $A=U D U^{\ast}$ and $B=Z Z^{\ast}/(2n)$. First, we consider a random matrix model for Theorem \ref{enhancedACM2} \eqref{B2}. We compute $\EV(A B+BA)$, $\EV(p A)\sqcup \EV(q A)$ and the theoretical limiting eigenvalues up to 6th. 
In this case, we obtained $p=2.4174$ and $q=-0.41351$. 
We also compute moments of $AB+BA$ 
with respect to the non-normalized trace $\Tr_n$ and factorized moments arising from cyclic monotone independence.

Second, we consider a random matrix model for Theorem \ref{enhancedACM2} \eqref{C2}. We compute $\EV(\ri(AB-BA))$, $\EV(r A)\sqcup \EV(-r A)$ and the theoretical limiting eigenvalues up to 10th. In this case, we have $r=0.98972$. We also compute the moments and factorized moments.
\begin{figure}[htbp]
\begin{center}
\includegraphics[width=7cm]{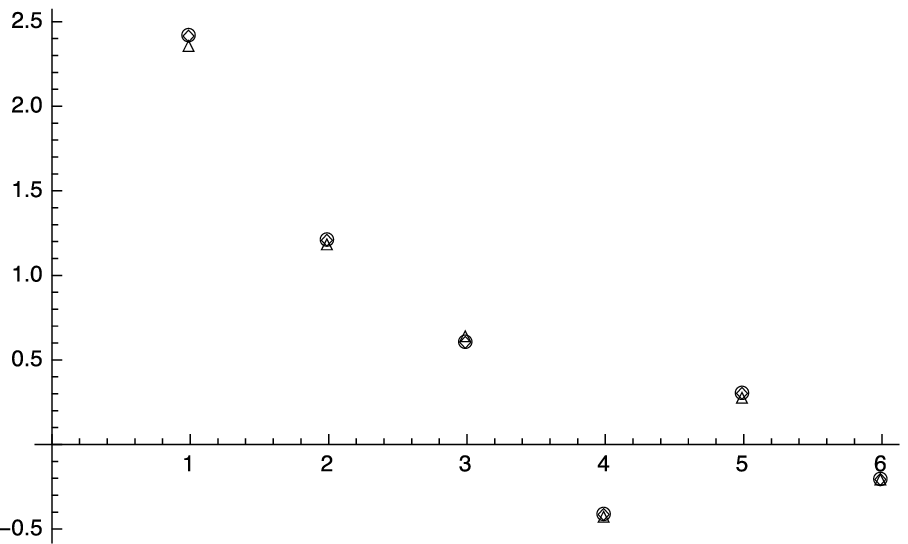}
\quad
\includegraphics[width=7cm]{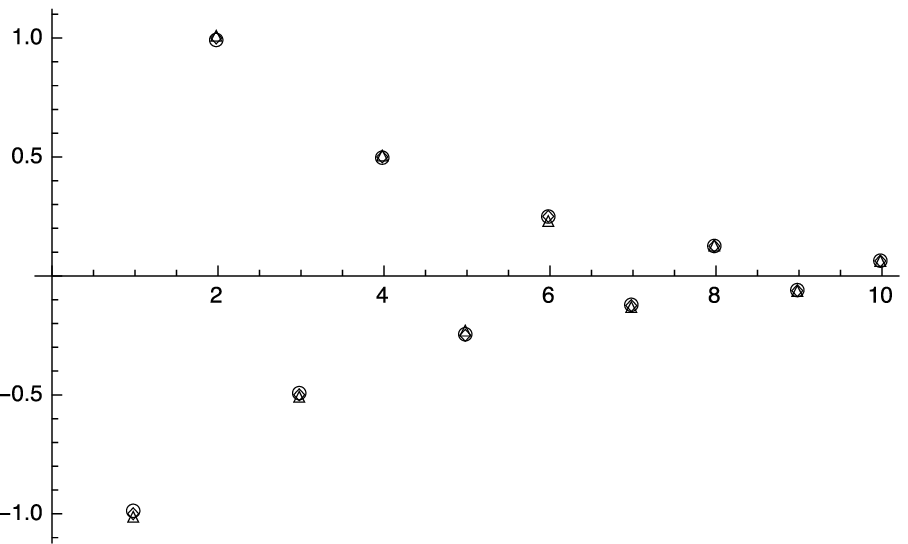}
\caption{Left: comparing $\EV(A B+BA)$ (triangle), $\EV(p A)\sqcup \EV(q A)$ (circle) and the theoretical limiting eigenvalues (square). 
Right: comparing $\EV(\ri(AB-BA))$ (triangle), 
$\EV(r A)\sqcup \EV(-r A)$ (circle) and the theoretical limiting eigenvalues (square).}
\end{center}
\label{fig2}
\end{figure}

\begin{table}[htbp]
\caption{Moments $(n=300)$}
\begin{center}
\begin{tabular}{|c|c||c|}
\hline
$k$&$\mathrm{Tr}((AB+BA)^{k})$ &$\mathrm{Tr}(A^{k})(p^{k}+q^{k})$\\
\hline
1 & 3.91122 &  4.00779  \\ \hline
2 & 7.71417   & 8.01973  \\ \hline
3 & 14.9663  & 16.0641   \\
\hline
\end{tabular}
\begin{tabular}{|c|c||c|}
\hline
$k$&$\mathrm{Tr}((i(AB-BA))^{2k})$ &$\mathrm{Tr}(2A^{2k}) r^{2k}$\\
\hline
1 & 2.70837 &  2.61211  \\ \hline
2 & 2.22769  & 2.04693 \\ \hline
3 & 2.16952  & 1.90957 \\
\hline
\end{tabular}
\end{center}
\label{moments2}
\end{table}

\end{example}

\section*{Acknowledgements} 
All authors were supported by JSPS KAKENHI Grant Number 26800048, 15K17549 and 15K04923, respectively.
B. Collins was supported by NSERC discovery and accelerator grant, and
ANR grant SToQ.


\end{document}